\documentclass[11pt,leqno,oneside]{amsart}
\usepackage{layout}
\usepackage{mathrsfs,dsfont}
\usepackage{amsmath,amstext,amsthm,amssymb,bbm,color}
\usepackage{comment}
\usepackage{charter}
\usepackage{typearea}
\usepackage{pdfsync}
\usepackage[width=6.5in,height=8.7in]{geometry}

\def\bt{\begin{thm}}
\def\et{\end{thm}}
\def\bl{\begin{lem}}
\def\el{\end{lem}}
\def\bd{\begin{defn}}
\def\ed{\end{defn}}
\def\bc{\begin{cor}}
\def\ec{\end{cor}}
\def\bp{\begin{proof}}
\def\ep{\end{proof}}
\def\br{\begin{rem}}
\def\er{\end{rem}}

\newtheorem{thm}{Theorem}[section]
\newtheorem{prop}[thm]{Proposition}
\newtheorem{lem}[thm]{Lemma}
\newtheorem{defn}[thm]{Definition}

\newtheorem{rem}[thm]{Remark}
\newtheorem{cor}[thm]{Corollary}

\numberwithin{equation}{section}

\newcommand{\bthm}{\begin{thm}}
\newcommand{\ethm}{\end{thm}}
\newcommand{\bstp}{\begin{stp}}
\newcommand{\estp}{\end{stp}}
\newcommand{\blemma}{\begin{lemma}}
\newcommand{\elemma}{\end{lemma}}
\newcommand{\bprop}{\begin{prop}}
\newcommand{\eprop}{\end{prop}}
\newcommand{\bpf}{\begin{pf}}
\newcommand{\epf}{\end{pf}}
\newcommand{\bdefn}{\begin{defn}}
\newcommand{\edefn}{\end{defn}}
\newcommand{\brk}{\begin{rmrk}}
\newcommand{\erk}{\end{rmrk}}
\newcommand{\bcrl}{\begin{crl}}
\newcommand{\ecrl}{\end{crl}}

\title[EQUIDISTRIBUTION]{Random systems of holomorphic sections of a sequence of line bundles on compact K\"{a}hler manifolds}

\address{Faculty of Engineering and Natural Sciences, Sabanc{\i} University, \.{I}stanbul, Turkey}

\email{afrimbojnik@sabanciuniv.edu}

\email{ozangunyuz@sabanciuniv.edu}


\begin{document}
\author{Afrim Bojnik and Ozan Günyüz}
\thanks{A.\ Bojnik was partially supported by T\"{U}B\.{I}TAK grant ARDEB-3501/118F049 and Tosun Terzioğlu Postdoctoral Fellowship.}

\begin{abstract}

This paper primarily establishes an asymptotic variance estimate for smooth linear statistics associated with zero sets of systems of random holomorphic sections in a sequence of positive Hermitian holomorphic line bundles on a compact Kähler manifold $(X, \omega)$ in a general non-Gaussian setting. Using this variance estimate and the expected distribution, we derive an equidistribution result for zeros of these random systems, which proves that the smooth positive closed form $\omega^{k}$ can be approximated by currents of integration along analytic subsets of $X$ of codimension $k$, $k \in \{1, \ldots, n\}$. The probability measures taken into consideration in this paper are sufficiently general to include a wide range of the measures commonly encountered in the literature, for which we give equidistribution results at the end, such as the standard Gaussian measure, Fubini-Study measure, the area measure of spheres, probability measures whose distributions have bounded densities with logarithmic decaying tails and locally moderate measures among others.

\end{abstract}
\maketitle

%%%%%%%%%%%%%%%%%%%%%%%%%%%%%%%%%%%%%%%%%%%%%%%%%%%%%%%%%%%%%%%%%%%%%%%%%%%%%%%

\section{Introduction}

In recent years, the equidistribution and statistical behavior of zeros of random holomorphic sections have been studied extensively. There are numerous results as to the distribution of zeros of holomorphic sections in diverse probabilistic backgrounds. Amongst these, what has been more largely focused on is the tensor powers of a given positive Hermitian line bundle over a compact (or non-compact) Kähler manifold within a Gaussian setting. In this background, \cite{SZ99} is one of the first papers considering the equidistribution problem of (Gaussian) random holomorphic sections. In later studies \cite{SZ08, SZ10}, asymptotic variance formulas are obtained not only for linear statistics but for their smooth analogs as well, all within the same geometric and probabilistic context. \cite{SZ10} also deals with the central limit theorem in the above complex geometric background. In the present setting of a sequence of line bundles, the authors of the current paper also obtained a more general asymptotic normality inspired by \cite{SZ10}, see \cite{BG1} for more details. One of the most recent results, proved via the techniques of \cite{SZ10} in \cite{Shif1}, is the asymptotic expansion of the variance for the codimension \(1\) case in the aforementioned setting. This asymptotic expansion also shows that the coefficient of the first term in the expansion, which also appeared as the leading-order term in the asymptotic formula proved in \cite{SZ10}, is sharp. In this framework, Dinh and Sibony \cite{DS06} introduced a method from complex dynamics for analyzing zero distribution, and set convergence speed bounds in the compact case, enhancing Shiffman and Zelditch's initial results, namely \cite{SZ99}. More recent studies, such as those by Dinh, Marinescu, and Schmidt \cite{DMS12}, and Drewitz, Liu, and Marinescu \cite{DLM23}, \cite{DLM1} have extended the equidistribution problem to non-compact complex manifolds, providing significant insights into the behavior of zeros in these wider settings (see also \cite{BG24} for related results in this non-compact framework). Alongside the Gaussian setting, in the papers \cite{Bay16}, \cite{BCHM}, \cite{BloomL}, \cite{CLMM}, \cite{CM1}, \cite{CMM}, \cite{G1} and \cite{BCM}, more general scenarios are investigated, including the Gaussian case as a particular instance. For example, in \cite{BloomL}, the authors focus on the complex random variables that possess bounded distribution functions on the whole complex plane \( \mathbb{C} \) and outside of a very large disk with radius \( \rho \), its integral with respect to the two-dimensional Lebesgue measure has an upper bound depending on \( \rho \); the latter condition is called the tail-end estimate. Meanwhile, in \cite{BCM, CMM} the authors expand their research to the equidistribution problem within an extensive context, involving a sequence of Hermitian line bundles over a normal reduced complex Kähler space with singularities.

It is essential to highlight that global holomorphic sections are natural generalizations of polynomials. In relation to the general setting, there has been a great deal of interest in the statistical problems related to the zero sets of random polynomials in several variables in both real and complex domains.  For a comprehensive overview of results in this direction, interested readers can refer to \cite{BloomL, BloomS, BL05, BloomD, EDKO, ROJ, SHSM, HN08, IZ, Gnyz1} (and references cited therein). These sources cover a wide range of results including both Gaussian and non-Gaussian cases along with historical developments of the polynomial theory. Long before these developments, it is worth acknowledging the pioneering work of mathematicians such as Littlewood-Offord, Kac, Hammersley, and Erd\"{o}s-Turan, who were among the first to investigate the distribution of roots of random algebraic equations in a single real variable. For more insights into these foundational studies, interested readers can consult the papers \cite{KAC, LO, HAM56, ET}.

On the other hand, there is a growing body of physics literature addressing equidistribution and probabilistic problems related to the zeros of complex random polynomials. For studies of fundamental importance in this area, see, for instance, \cite{NV98}, which slightly predates \cite{SZ99} and can be regarded as a foundational work on the zeros of holomorphic sections in the specific context of theta bundles over an elliptic curve $\mathbb{C}/\mathbb{Z}^{2}$.

The present article follows a general setting akin to papers \cite{CLMM} and \cite{CMM}. In the recent work presented by Coman-Lu-Ma-Marinescu \cite{CLMM}, they establish an equidistribution result for a sequence of line bundles $(L_{p}, h_{p})$, instead of the tensor powers of a single line bundle $L$, i.e., $(L^{\otimes p}, h^{\otimes p})$, by imposing a natural convergence condition on the Chern curvature forms $c_{1}(L_{p}, h_{p}).$ As a probability measure they consider the  Fubini-Study measures and use the standard formalism of meromorphic transforms from complex dynamics, as introduced by Dinh and Sibony in \cite{DS06}.

 In contrast to their work, the novelty of the current paper lies in its reliance on the variance estimate of zero currents of integration over the zero sets of systems of random holomorphic  sections for any codimension. This classical but efficient approach allows for the extension of the previous results to a broader spectrum of probability distributions, generalizing Theorem 0.4 considered in \cite{CLMM} as well. It is worth noting that, in the context of sequences of line bundles, the variance estimate is proven for the first time in this work setting it apart from previous related results in the literature. Furthermore, our findings can be viewed as a form of universality result within a wide range of probability measures studied in the literature, including those with continuous distributions.

We consider a sequence of holomorphic line bundles \( (L_{p}, h_{p})_{p \geq 1} \) on a compact Kähler manifold \( (X, \omega) \) of complex dimension $n = \text{dim}_{\mathbb{C}} \, X$ with a complex structure $J$ on $X$, where $\omega$ is a fixed Kähler form on $X$ and $h_{p}$ are \( \mathscr{C}^{2} \) Hermitian metrics (see Section 2 below). We assume that the curvature forms \( c_{1}(L_{p}, h_{p}) \) satisfy the following \textit{diophantine approximation} condition:
\begin{itemize}
	\item [\textbf{(A)}] There exists a sequence \( A_{p} > 0 \) with \( \lim_{p \rightarrow \infty} A_{p} = +\infty \) and a constant \( a > 0 \) such that
	\begin{equation}\label{cur}
		\left \| \frac{1}{A_{p}} c_{1}(L_{p}, h_{p}) - \omega \right \|_{\mathscr{C}^{0}} = O(A_{p}^{-a}).
	\end{equation}
\end{itemize}
Let $g^{TX}(u, v)= \omega(u, Jv)$, $u, v \in TX$, be the Riemannian metric on $TX$ induced by $\omega$ and $J$. The $k$-volume form \( \mathrm{Vol}_k \) is expressed as \( \mathrm{Vol}_{k}(A) = \int_{A}{\omega^{n-k}} \). We suppress the subindex because it will be clear from the context which codimension is meant. We also remark that, unlike \cite{CLMM}, for the sake of simplifying our notation, we will not utilize another volume form on the base manifold $X$ besides $\omega$. Even though employing a different form than $\omega$ might alter the notation, it will not affect the equidistribution results of this paper. The adjustment mainly involves substituting the appropriate powers of $\omega$ with another form $\vartheta$ in the relevant parts, such as basic cohomology arguments and the total variation of the signed measure $dd^{c} \phi$ for a test form $\phi$.

In the standard setup of geometric quantization, one works with a compact Kähler manifold $X$ with a fixed K\"{a}hler form $\omega$ on it, denoted by $(X, \omega)$, equipped with a Hermitian holomorphic line bundle $(L, h)$, known as the prequantum line bundle, fulfilling the prequantization condition given by 
\begin{equation}\label{preq}\omega=\frac{\sqrt{-1}}{2\pi} R^{L}=c_{1}(L, h).\end{equation} 
Here $R^L$ is the curvature of the Chern connection on $L$, and $c_1(L,h)$ is the Chern curvature form of $(L,h)$. The prequantum line bundle $(L,h)$ provides the geometric framework for quantizing the Kähler manifold $(X,\omega)$. The corresponding space of quantum states is the Hilbert space $H^0(X,L)$. Different quantization levels are obtained by considering the tensor powers $(L^{\otimes p},h^{\otimes p})$, and the limit $p\to\infty$ corresponds to the semiclassical regime, with effective Planck constant $\hbar=1/p$.

The condition \eqref{preq} is an integrality condition on the Kähler class $[\omega]\in H^2(X,\mathbb{R})$. If $[\omega]$ is integral, i.e. if it lies in the image of the natural map $H^2(X,\mathbb{Z})\to H^2(X,\mathbb{R})$, then $\omega$ can be realized as the Chern form of a prequantum holomorphic line bundle. If $[\omega]$ is not integral, such a realization is not available in general. More generally, if one wants to approximate $[\omega]$ by normalized first Chern classes of holomorphic line bundles, then $[\omega]$ has to belong to the real Néron--Severi subspace; for a precise discussion, see Demailly \cite{Demailly2010}. Outside this subspace, one can in general only approximate using smooth complex line bundles, which are generally not holomorphic.

Accordingly, in the present paper, condition \eqref{cur} is taken as an assumption on the existence of a sequence of positive Hermitian holomorphic line bundles $(L_p,h_p)$ whose normalized Chern forms approximate $\omega$. This is the geometric setting considered in \cite{CLMM}. In this setting, \cite{CLMM} proves a full asymptotic expansion of the Bergman kernel restricted to the diagonal, and in the present work we prove an equidistribution result for the zeros of systems of random holomorphic sections. For the equidistribution problem, following \cite{CLMM}, we work in a setting as general as possible, where the metrics are of class $\mathscr{C}^{2}$ and the convergence of Chern forms in \eqref{cur} is in the $\mathscr{C}^{0}$-topology.

It is also important to note that, within the examples of sequences of line bundles $(L_{p}, h_{p})$ fulfilling the condition \eqref{cur}, one natural instance is $(L_{p}, h_{p})=(L^{\otimes p},h^{\otimes p})$ for some fixed prequantum line bundle $(L,h)$. Other examples include cases where $(L_{p}, h_{p})=(L^{\otimes p},h_{p})$ but here, $h_{p}$ is not necessarily the product metric $h^{p}$, e.g. $h_{p}=h^{p}e^{-\varphi_{p}}$ with appropriate weights $\varphi_{p}$. For examples involving tensor powers of several line bundles, see \cite{CLMM}.

We denote the vector space of global holomorphic sections of $L_p$ by $H^{0}(X, L_{p})$. We take into consideration the following inner product on the space of smooth sections $\mathscr{C}^{\infty}(X,L_{p})$  with respect to the metric $h_{p}$ and the volume form $\omega^{n}$ on $X$:
\begin{equation*}
\big\langle s_{1}, s_{2} \big\rangle_{p}:=\int_{X}{\big \langle s_{1}(x), s_{2}(x) \big \rangle_{h_{p}} \omega^{n}}\, \,\, \,\text{and}\,\,\, \|s\|^{2}_{p}:=\big\langle s, s \big\rangle_{p}.\end{equation*}
By virtue of Cartan-Serre finiteness theorem (see, e.g., chapter 6, \cite{GR3}), the space $H^{0}(X, L_{p})$ is finite dimensional and we will write $d_{p}:= \dim{H^{0}(X, L_{p})}$. The space $\mathcal{L}^{2}(X, L_{p})$ represents the completion of $\mathscr{C}^{\infty}(X,L_{p})$ with respect to this norm, forming the Hilbert space of square integrable sections of $L_{p}$. Consider the orthogonal projection operator  $\Pi_{p}: \mathcal{L}^{2}(X, L_{p})\rightarrow H^{0}(X, L_{p}).$  The Bergman kernel, denoted as $K_{p}(x, y)$ is defined as the integral kernel associated to this projection (for details see, e.g., \cite{MM1}, section $1.4$). If $\{S^{p}_{j} \}_{j=1}^{d_{p}}$ is an orthonormal basis for $H^{0}(X, L_{p}),$ by using the orthonormal representation of $K_{p}(x, y)$ with respect to the basis $\{S^{p}_{j} \}_{j=1}^{d_{p}}$ and making use of the reproducing property on the space $H^{0}(X, L_{p})$ we have: $$K_{p}(x, y)=\sum_{j=1}^{d_{p}}{S^{p}_{j}(x)\otimes S^{p}_{j}(y)^{*}}\in L_{p, x}\otimes L_{p, y}^{*},$$where $S^{p}_{j}(y)^{*}=\langle \, .\,, S^{p}_{j}(y)\rangle_{h_{p}}\in L^{*}_{p, y}$. The restriction of the Bergman kernel to the diagonal of \( X \) is called the Bergman kernel function of \( H^{0}(X, L_{p}) \), which we denote by \( K_{p}(x) := K_{p}(x, x) \). This kernel function has the dimensional density property, meaning that \( d_{p} = \int_{X} K_{p}(x) \omega^{n} \). We now make the following assumption on the behavior of \( K_{p}(x) \):

	\textbf{(B)} There exists a constant \( M_{0} > 1 \) and \( p_{0} \in \mathbb{N} \) such that
	\begin{equation}\label{0.11}
		\frac{A_{p}^{n}}{M_{0}} \leq K_{p}(x) \leq M_{0} A_{p}^{n},
	\end{equation}
	for every \( x \in X \) and \( p \geq p_{0} \).

This assumption produces the following estimates on the dimension \( d_{p} \), which will be useful in the upcoming analysis:
\begin{equation}\label{dim}
	\frac{\mathrm{Vol}(X) A_{p}^{n}}{M_{0}} \leq d_{p} \leq M_{0} \mathrm{Vol}(X) A_{p}^{n},
\end{equation}
for all \( p \geq p_{0} \).

We also make another simple observation that will be useful in the fourth section. Choosing $p'\in\mathbb{N}$ sufficiently large so that $A_{p}\geq M_{0}$ so that the assumption (\ref{0.11}) can be written as follows \begin{equation}\label{m1.3} A_{p}^{n-1}\leq K_{p}(x)\leq A_{p}^{n+1}\end{equation} for all such $p \geq p'$.

To introduce randomness, let us fix an orthonormal basis \( \{S^{p}_{j}\}^{d_{p}}_{j=1} \) of \( H^{0}(X, L_{p}) \). Then, each \( s_{p} \in H^{0}(X, L_{p}) \) can be uniquely expressed as:
\begin{equation}\label{rep}
	s_{p} = \sum_{j=1}^{d_{p}}{a_{j}^{p} S^{p}_{j}}.
\end{equation}

Using this representation, we identify the space \( H^{0}(X, L_{p}) \) with \( \mathbb{C}^{d_{p}} \) and equip it with the \( d_{p} \)-fold probability measure \( \sigma_{p} \), which satisfies the following conditions:

\begin{enumerate}
	\item [\textbf{(C1)}] The measure \( \sigma_{p} \) does not charge pluripolar sets.
	\item [\textbf{(C2)}] (Moment condition) There exists a constant \( \alpha \geq 2 \) and, for each \( p \geq 1 \), constants \( C_{p} > 0 \) such that
	$$
	\int_{\mathbb{C}^{d_{p}}} \big| \log |\langle a, v \rangle| \big|^{\alpha} d\sigma_{p}(a) \leq C_{p},
	$$
	for every \( v \in \mathbb{C}^{d_{p}} \) with \( \|v\| = 1 \).
\end{enumerate}

The probability space \( (H^{0}(X, L_{p}), \sigma_{p}) \) depends on the choice of the orthonormal basis used in the identification of \( H^{0}(X, L_{p}) \), unless \( \sigma_{p} \) is unitary invariant. Additionally, we consider the product probability space \( (\mathcal{H}_{\infty}, \sigma_{\infty}) = \left(\prod_{p=1}^{\infty} H^{0}(X, L_{p}), \prod_{p=1}^{\infty} \sigma_{p}\right) \), which consists of random sequences of global holomorphic sections of \( L_{p} \) for increasing values of \( p \).

Similarly, for the central objectives of this paper, we consider codimensions greater than $1$ by examining the product probability spaces $(H^{0}(X, L_{p})^{k}, \sigma_{p}^{k})$, where $\sigma_{p}^{k} = \sigma_{p} \times \cdots \times \sigma_{p}$ denotes the $k$-fold product measure on the space $H^{0}(X, L_{p})^{k}$. In this context, we take sequences in $H^{0}(X, L_{p})^{k}$ (i.e., sequences of random systems of global holomorphic sections) into account by considering the associated infinite product probability space, denoted as
 $$(\mathcal{H}^{k}_{\infty}, \sigma^{k}_{\infty}) = \left(\prod_{p=1}^{\infty} H^{0}(X, L_{p})^{k}, \prod_{p=1}^{\infty} \sigma^{k}_{p}\right).$$

Unlike \cite{BCM}, we have two probabilistic conditions on $H^{0}(X, L_{p})$. The condition \textbf{(C1)} is necessary and crucial for higher codimensional analysis, and is needed in the proof of the probabilistic version of Bertini’s theorem (see Proposition \ref{bertiniseveral} and Proposition \ref{wdberti} in Appendix) to ensure proper intersections of zero sets of each random holomorphic section as a component of a given random system. Obviously, it is not used when considering only a single holomorphic section, as was investigated in \cite{BCM}. The condition \textbf{(C2)} is a slightly different moment condition compared to the one given in \cite{BCM}, that is we take $\alpha \geq 2$ (see, e.g., p.3, assumption (B) in \cite{BCM}, where $\alpha \geq 1$ was the condition for the exponent $\alpha$). This alteration plays a key role in determining variance bounds, for further details, see Section \ref{S2}.

$\mathcal{D}^{p, q}(X)$ denotes the space of test forms of bidegree $(p, q)$ on the complex manifold $X$ and  we will let $\mathcal{D'}_{p, q}(X)$ denote the space of currents of bidegree $(p, q)$ on $X,$ so $\langle T, \varphi\rangle=T(\varphi)$ will mean the pairing of \,$T\in \mathcal{D'}_{p, q}(X)$ and $\varphi\in \mathcal{D}^{n-p, n-q}(X)$. We will work with only \textbf{real-valued} test forms for simplicity, (necessarily) of bidegree $(p, p)$. For a thorough investigation of currents in the context of complex manifolds, see, e.g.,  \cite{Dem12}.

Given a system \( \Sigma_{p}^{k}:= (s^{1}_{p}, \ldots, s^{k}_{p}) \) of \( k \) holomorphic sections of \( L_{p} \), where \( 1 \leq k \leq \dim X \), we denote their simultaneous zero locus by
$$
Z_{\Sigma_{p}^{k}} := \{ x \in X : s_{p}^{1}(x) = \cdots = s_{p}^{k}(x) = 0 \}.
$$
Its current of integration (with multiplicities) along the analytic subvariety \( Z_{\Sigma_{p}^{k}} \) is defined as follows. For \( \phi \in \mathcal{D}^{n-k, n-k}(X) \), we have
$$
\langle [Z_{\Sigma_{p}^{k}}], \phi \rangle := \int_{\operatorname{Reg}(Z_{\Sigma_{p}^{k}})} \phi,
$$
where \( \operatorname{Reg}(Z_{\Sigma_{p}^{k}}) \) denotes the set of regular points of the simultaneous zero locus \( Z_{\Sigma_{p}^{k}} \). The base locus is defined as \( \operatorname{Bs}(H^{0}(X, L_{p})):=\{x\in X: s(x)=0 \,\,\text{for} \,\,\text{all}\,\, s\in H^{0}(X, L_{p}) \} \). By \textbf{(B)}, the base locus \( \operatorname{Bs}(H^{0}(X, L_{p})) \) is empty for \( p \geq p_{0} \). Consequently, for $\sigma_{p}^{k}$-almost every system $\Sigma_{p}^{k}\in H^{0}(X, L_{p})^{k}$, the simultaneous zero set $Z_{\Sigma_{p}^{k}}$ is achieved as a smooth complete intersection of the zero loci of individual sections. Moreover, the current of integration over $Z_{\Sigma_{p}^{k}}$ is given by $$[Z_{\Sigma_{p}^{k}}]:=[Z_{s_{p}^{1}}]\wedge \cdots \wedge [Z_{s_{p}^{k}}]$$
for $\sigma^{k}_{p}$-almost all $\Sigma_{p}^{k}$. See Section 3 for details. We will call such a probabilistically generic element \textit{typical}.

\vspace{5mm}

 The expectation and the  variance of the current valued random variable $(H^{0}(X,L_{p})^{k}, \sigma_{p}^{k})\ni \Sigma_{p}^{k}\longmapsto [Z_{\Sigma_{p}^{k}} ]$ are defined by \begin{align}\label{expvari}
\mathbb{E} \langle[Z_{\Sigma_{p}^{k}}], \phi \rangle &:=\int_{H^{0}(X, L_{p})^{k}}{\langle [Z_{\Sigma_{p}^{k}}], \phi\rangle\, d\sigma_{p}^{k}(\Sigma_{p}^{k})} \\  \mathrm{Var} \langle [Z_{\Sigma_{p}^{k}}], \, \phi \rangle &:= \mathbb{E}\langle [Z_{\Sigma_{p}^{k}}], \, \phi \rangle^{2}- (\mathbb{E}\langle [Z_{\Sigma_{p}^{k}}], \, \phi \rangle)^{2},\end{align} where $\phi\in \mathcal{D}^{n-k, n-k}(X)$.
 Consistent with the diophantine approximation (\ref{cur}), we will consider normalized (up to the volume of $X$ times some dimensional constant) currents of integration along the zero set $Z_{\Sigma_{p}^{k}}$ of the random system $\Sigma_{p}^{k} = (s^{1}_{p}, \ldots, s^{k}_{p})$, more precisely,
  \begin{equation*}
  	\big[\widehat{Z}_{\Sigma_{p}^{k}}\big]:=\frac{1}{A_{p}^{k}}\big[Z_{\Sigma_{p}^{k}}\big].
  \end{equation*}

 Now, we are ready to give the main results of this paper:

\begin{thm}\label{th1}
	Let \( (L_{p}, h_{p})_{p \geq 1} \), \( (X, \omega) \), and \( \sigma_{p} \) be as defined above. Assume that they satisfy the conditions \textbf{(A)}, \textbf{(C1)}-\textbf{(C2)},
and the linear system $H^{0}(X, L_p)$ is base-point free, i.e. $\operatorname{Bs}(H^{0}(X, L_p))=\varnothing$ for large values of $p$. Further, let the systems \( \Sigma^{k}_{p}= (s_{p}^{1}, s_{p}^{2}, \ldots, s_{p}^{k}) \) with sections chosen independently with respect to \( \sigma^{k}_p \) be given. Then, there exists \( P \in \mathbb{N} \) such that for all \( p \geq P \) and any \( \phi \in \mathcal{D}^{n-k, n-k}(X) \), the following estimate holds:
	\begin{equation}\label{mainvar}
			\mathrm{Var}\langle [\widehat{Z}_{\Sigma_{p}^{k}}], \phi \rangle  \leq D_{n} \frac{(C_{p})^{2/\alpha}}{A^{2}_{p}}\,(B_{\phi})^{2}(\mathrm{Vol}(X))^{2}
			 \end{equation}where  $B_{\phi}= b \|\phi\|_{\mathscr{C}^{2}}$ is a positive constant depending on the form $\phi$, and $D_{n}$ is a constant that depends only on the dimension of $X$.
\end{thm}

Although we do not focus on deriving the asymptotic formula for the variance in this paper, we wish to mention some important results in this direction from the literature. In their article \cite{SZ08}, Shiffman and Zelditch, in Section 3, first define the variance current in the prequantum setting, associated with the current of integration over the zero set of a system of holomorphic sections. They then prove an explicit formula for this current (Theorem 3.13), particularly for the pair correlation current, using a so-called pluribipotential function, in the standard Gaussian setting. In \cite{SZ10}, this result from \cite{SZ08} is used to derive an explicit formula (Corollary 3.3) for the variance of smooth linear statistics associated with a holomorphic system. One of the key advantages of the standard Gaussian setting is that, in codimension $1$, the variance current admits an alternative representation via a pluribipotential function (In particular, \cite[lemma 3.3]{SZ08} provides a useful property for expectation relations of complex joint Gaussian random variables). Specializing our framework to the prequantum line bundle case with $(L_{p}, h_{p})=(L^{\otimes p}, h^{\otimes p})$ and $A_p=p$, our general argument yields an $O(p^{-2})$ bound for the variance of the normalized zero current, which they also obtained in their paper \cite{SZ99}, whereas the Gaussian case considered by Shiffman-Zelditch's asymptotic result (\cite[Theorem 1.1]{SZ10}) provides the order $O(p^{-(n+2)})$, which is obviously sharper.

\begin{thm}\label{th2}
	Let \( (L_{p}, h_{p})_{p \geq 1} \), \( (X, \omega) \), and \( \sigma_{p} \) be as defined above. Assume that they satisfy the conditions \textbf{(A)}, \textbf{(B)}, and \textbf{(C1)}-\textbf{(C2)}.
	
	\begin{itemize}
		\item [(i)] If \( \lim_{p \to \infty} \frac{C_{p}^{1/\alpha}}{A_{p}} = 0 \), then for \( 1 \leq k \leq \dim_{\mathbb{C}} X \),
		$$
		\mathbb{E}\big[\widehat{Z}_{\Sigma_{p}^{k}}\big] \longrightarrow \omega^k
		$$
		in the weak* topology of currents as \( p \to \infty \).
		
		\item [(ii)] If \( \sum_{p=1}^{\infty} \frac{C_{p}^{2/\alpha}}{A_{p}^{2}} < \infty \), then for \( \sigma_{\infty}^{k} \)-almost every sequence \( \{\Sigma_{p}^{k}\} \in \mathcal{H}_{\infty}^{k} \),
		$$
		\big[\widehat{Z}_{\Sigma_{p}^{k}}\big] \longrightarrow \omega^k
		$$
		in the weak* topology of currents as \( p \to \infty \).
	\end{itemize}
\end{thm}

If the measure \( \sigma_{p} \) satisfies the moment condition \textbf{(C2)} with constants \( C_{p} = \Lambda \), independent of \( p \), then the assumption in (i), \( \lim_{p \to \infty} C_{p}^{1/\alpha} A_{p}^{-1} = 0 \), is automatically satisfied. Furthermore, hypothesis (ii) reduces to \( \sum_{p=1}^{\infty} A_{p}^{-2} < \infty \). \cite[Theorem 0.4]{CLMM}, which was proved by using the meromorphic transforms from complex dynamics for Fubini-Study probability measures, is one of the examples of this situation, that is to say, it is a special case of Theorem \ref{th2}, see Section \ref{sec5} for this case and more other probability measures satisfying \textbf{(C2)}.

For codimension one, the approach employed by Bayraktar, Coman, and Marinescu \cite{BCM}—which circumvents the use of variance and expected distribution in the setting where the first Chern currents may not converge—works effectively. However, this method does not extend to codimensions greater than one in the same setting. Our approach, utilizing the convergence condition \eqref{cur} on the Chern forms, generalizes the results of equidistribution in higher codimensions in the present setting of line bundle sequences.

At the same time, our principal result, which has been established for codimension $k$, is applicable to a multitude of frequently investigated probability measures. These are the area measure of spheres, Gaussian measure, Fubini-Study measure, and measures with bounded density having logarithmically decaying tails. When our main result is applied to these measures in the context of tensor powers of a fixed prequantum line bundle, the required summability assumption can be dropped. Our methods extend the results of Shiffman and Zelditch (\cite{SZ99}).

\section{Background}\label{S2}

Let \((X, \omega)\) be a compact Kähler manifold with \(\dim_{\mathbb{C}}{X}=n\).
A holomorphic line bundle \(L\) over $X$ is defined by compiling complex lines \(\{L_{x}\}_{x\in X}\)
and constructing a complex manifold of dimension \(1 + \dim_{\mathbb{C}}{X}\) with a projection map
\(\pi : L \rightarrow X \) such that \(\pi\) assigns each line (or fiber) \(L_{x}\) to \(x\) and is holomorphic.
Using an open cover \(\{U_{\alpha}\}\) of \(X\), we can locally trivialize \(L\) through biholomorphisms
\(\Psi_{\alpha}: \pi^{-1}(U_{\alpha}) \rightarrow U_{\alpha} \times \mathbb{C}\) which map \(L_{x} = \pi^{-1}(x)\)
isomorphically onto \(\{x\} \times \mathbb{C}\). The line bundle \(L\) is then uniquely (i.e., up to isomorphism)
determined by these transition functions \(g_{\alpha \beta}\), which are non-vanishing holomorphic functions
on \(U_{\alpha \beta}:=U_{\alpha} \cap U_{\beta}\) defined by \(g_{\alpha \beta}= \Psi_{\alpha} \circ \Psi^{-1}_{\beta}|_{\{x\} \times \mathbb{C}}\).
These functions \(g_{\alpha \beta}\) satisfy the cocycle condition \(g_{\alpha \beta} g_{\beta \gamma} g_{\gamma \alpha}=1\).

Because the transition functions satisfy the cocycle condition, they define a cohomology class, denoted as
\([g_{\alpha \beta}] \in H^{1}(X, \mathcal{O}^{*})\). Here, \(H^{1}(X, \mathcal{O}^{*})\) is the first sheaf
cohomology group of the manifold \(X\) with coefficients in the sheaf of non-zero holomorphic functions,
denoted by \(\mathcal{O}^{*}\). The exponential short exact sequence \(0 \rightarrow \mathbb{Z} \rightarrow \mathcal{O}
\rightarrow \mathcal{O}^{*} \rightarrow 0\) produces a mapping \(c_{1}: H^{1}(X, \mathcal{O}^*) \rightarrow H^{2}(X, \mathbb{Z})\) and the first Chern class
\(c_{1}(L)\) is defined as the image of \([g_{\alpha \beta}]\) under this mapping.

Let $\{U_{\alpha}\}$ be an open covering of $X$ and $e_{\alpha}$ be  holomorphic frames of $L$ on each $U_{\alpha}$. Then the metric is given by a family of functions $h_{\alpha}=|e_{\alpha}|^{2}_{h}:U_{\alpha}\rightarrow [0, \infty]$ with local weights $\varphi_{\alpha}=-\frac{1}{2}\log{h_{\alpha}}\in \mathscr{C}^{2}(U_{\alpha})$. Given the transition functions $g_{\alpha\beta}=\frac{e_{\beta}}{e_{\alpha}}\in \mathcal{O}^{*}_{X}(U_{\alpha}\cap U_{\beta})$ of $L$, then on $U_{\alpha}\cap U_{\beta},$\, we have that $h_{\beta}=|g_{\alpha\beta}|^{2}h_{\alpha}$ and therefore $\log{h_{\beta}}=\log{|g_{\alpha\beta}|^{2}}+ \log{h_{\alpha}},$ which is equivalent to the following  $$\varphi_{\alpha}=\varphi_{\beta}+ \log{|g_{\alpha\beta}|}\,\, \text{in}\,\,\mathscr{C}^{2}(U_{\alpha}\cap U_{\beta}).$$

In this paper, the notion of positivity we use will be what is known in the literature as Nakano-Griffiths positivity for line bundles: A holomorphic line bundle $(L, h),$ with a Hermitian metric (that may not be $\mathscr{C}^{\infty}$)  is said to be positive or semi-positive if the local $\mathscr{C}^{2}$ weight functions $\varphi$ corresponding to the metric $h$ are strictly plurisubharmonic or plurisubharmonic, respectively.

Since we will be dealing with positive line bundles, the local weight functions $\varphi_{\alpha}$ above are strictly plurisubharmonic, which means that, in addition to being of class $\mathscr{C}^{2}$, $dd^{c} \varphi_{\alpha} >0$ holds pointwise. By Proposition 2.4 of \cite{CMM}, the line bundles we consider are ample and hence $X$ is projective.

We use the anticommutativity convention for a wedge product of a current $T$ of bidegree $(p, q)$ and a complex differential form $\theta$ of type $(r, s)$ (recall that $(T \wedge \theta)(\varphi):= T(\theta \wedge\varphi)$ for $\varphi \in \mathcal{D}^{n-p-r, \,n-q-s}(X)$), \begin{equation} \label{antc}\theta \wedge T:= (-1)^{(p+q)(r+s)} (T \wedge \theta).\end{equation} Given that all the objects we consider are of bidegree or type $(p, p)$, we will not distinguish between $\theta \wedge T$ and $T \wedge \theta$. This simplification will be helpful in the proof of Theorem \ref{exdist}.

The following fact (see section 4 of \cite{BCM}) will be required in the next section: There exists a constant $b>0$ such that for every  $\phi\in \mathcal{D}^{n-k, n-k}(X)$, $k \in \{1, 2, \ldots, n\}$, \begin{equation}\label{hercomp}
                              -b\, \| \phi \|_{\mathscr{C}^{2}}\, \omega^{n-k+1}\leq dd^{c}\phi \leq b \, \| \phi \|_{\mathscr{C}^{2}} \,  \omega^{n-k+1},
                            \end{equation}which amounts to saying that the total variation of $dd^{c}\phi$ verifies the inequality $|dd^{c}\phi|\leq b \, \| \phi \|_{\mathscr{C}^{2}} \, \omega^{n-k+1}$.

 In our local analysis throughout the paper, we always consider trivializing open neighborhoods of a given point $x\in X$. Let $U_{p}$ be such a trivializing neighborhood of $x$, and let $\{S_{1}^{p}, \ldots, S^{p}_{d_{p}}\}$ be an orthonormal basis for $H^{0}(X, L_{p})$. For a local holomorphic frame  $e_{p}$ of $L_{p}$ in $U_{p}$, we have $S^{p}_{j}=s^{p}_{j}e_{p},$\,where $s^{p}_{j}$ is a holomorphic function in $U_{p}.$ Then the Bergman kernel functions and the Fubini-Study currents are defined as follows:
\begin{equation}\label{bergfub}
  K_{p}(x)=\sum^{d_{p}}_{j=1}{|S^{p}_{j}(x)|^{2}_{h_{p}}}\,\,\,\text{and}\,\,\,\gamma_{p}|_{U_{p}}=\frac{1}{2}dd^{c}\log\big(\sum^{d_{p}}_{j=1}{|s^{p}_{j}|^{2}}\big),
\end{equation}
where $d=\partial+\overline{\partial},$ and $d^{c}=\frac{1}{2\pi i}(\partial- \overline{\partial})$.  It is important to emphasize that $K_{p}$ and $\gamma_{p}$ are independent of the chosen basis $\{{S_{1}^{p}, \ldots, S^{p}_{d_{p}}}\}$, (see \cite{CM1}, section 3). Moreover, $\gamma_{p}$ is a positive closed current of bidegree $(1, 1),$ smooth away from the base locus $\operatorname{Bs}(H^{0}(X,L_{p}))$ of $H^{0}(X,L_{p}),$ and by (\ref{bergfub}) we have that \begin{equation}\label{fubstu} \log{K_{p}}\in L^{1}(X,  \omega^{n})\,\,\, \text{and}\,\,\,\gamma_{p}- c_{1}(L_{p}, h_{p})=\frac{1}{2}dd^{c}\log{K_{p}},\end{equation}
which essentially means that $\gamma_{p}$  has the same de Rham cohomology class as $c_{1}(L_{p},h_{p})$.

 For large enough $p\in \mathbb{N}$, if $\Phi_{p}: X\dashrightarrow\mathbb{P}^{d_{p}-1}$ is the Kodaira map defined by the basis $\{S_{j}^{p}\}_{j=1}^{d_{p}}$ via $$\Phi_{p}(x)=[S_{1}^{p}(x):\ldots:S_{d_{p}}^{p}(x)],\,\,\,\,$$ then we have $\gamma_{p}=\Phi_{p}^{*}(\omega_{FS}),$ which justifies their name. This mapping $\Phi_{p}$ is also a holomorphic embedding by Corollary 1.3 of \cite{CLMM} whose proof can be carried out by the same arguments in \cite{GH} (p. 189), and thus it is a generalization of the classical Kodaira embedding theorem to the current setting.

Before finishing this section, let us give an important tool also known as Poincar\'{e}-Lelong formula (see \cite{MM1}, Theorem 2.3.3): Given $s\in H^{0}(X, L_{p}),$ we have \begin{equation}\label{lelpoi}
  [Z_{s}]= c_{1}(L_{p}, h_{p})+ dd^{c}\log{|s|_{h_{p}}}.
\end{equation}

%%%%%%%%%%%%%%%%%%%%%%%%%%%%%%%%%%%%%%%%%%%%%%%%%%%%%%%%%%%%%%%%%%%%%%%%%%%%%%%

\section{Variance Estimate}\label{S3} In this section, we will give the proof of Theorem \ref{th1} using an inductive approach based on the codimension $k$ by proving it in a more general setting of several line bundles inspired by \cite[Theorem 3.1]{Shif}. Theorem \ref{th1} will follow from this theorem.

Until we present the proof of Theorem \ref{th1} at the end of this section, we will temporarily adjust the notation for spaces of holomorphic sections and probability measures since we will be working with $k$ distinct holomorphic line bundles, $k\in \{1, \ldots, n\}$.

 For $j=1, \ldots, k$, let $\mathcal{S}_{j} \subset H^{0}(X, L_{j})$ be (finite dimensional) subspaces of dimension $d_{j}$ with their fixed orthonormal bases $\{S_{j, d_{j}}\}$ and identifications $\mathcal{S}_{j} \simeq \mathbb{C}^{d_{j}}$. For these subspaces we denote the associated Bergman kernels by $K_{\mathcal{S}_{p,j}}(x)$.
When we say that \textbf{(B)} and \textbf{(C2)} hold in this auxiliary setting, we mean that they hold with the same constants relative to the orthonormal bases and subspaces fixed above. Likewise, \textbf{(C1)} is understood to mean that the corresponding measures assign zero mass to pluripolar subsets of $\mathcal{S}_j$.

We will start with the case of codimension $1$, serving as the initial step in our induction process. To begin, let us make some initial observations.

Let $s_{p} \in \mathcal{S}_{p} \subset H^{0}(X, L_{p})$. We can write it as $$s_{p}=\sum_{j=1}^{d_{p}}{a_{j}^{p}S^{p}_{j}}=\langle a, \Gamma_{p} \rangle,$$ where $\Gamma_{p}=(S_{1}^{p}, \ldots, S^{p}_{d_{p}}),\,a=(a_{1}^{p}, \ldots, a_{d_{p}}^{p})\in \mathbb{C}^{d_{p}}$\, and\, $\{S^{p}_{j}\}_{j=1}^{d_{p}}$\, is an orthonormal basis of $\mathcal{S}_{p}$.

For \( x \in X \), let \( U \) be a small open neighborhood of \( x \) such that
for each \( p \), we take a local holomorphic frame \( e_{p} \) of \( L_{p} \) on \( U \). Then locally $S^{p}_{j}=f_{j}e_{p}$, where $f_{j}$ are holomorphic functions in $U$\, and so, by writing $f=(f_{1}, \ldots, f_{d_{p}}),$ $$s_{p}=\sum_{j=1}^{d_{p}}{a_{j}^{p}f_{j}e_{p}}=\langle a, f \rangle e_{p}.$$By Poincar\'{e}-Lelong formula (\ref{lelpoi}), on the neighborhood $U$, we have $$ [Z_{s_{p}}]=dd^{c}\log{|\langle a, f \rangle|}= dd^{c}\log|\langle a, \Gamma_{p}\rangle|_{h_{p}}+c_{1}(L_{p}, h_{p}).$$ Now, for any $\phi \in \mathcal{D}^{n-1, n-1}(X)$, we define the following random variable \begin{equation} \label{Wss} W_{s_{p}}:=[Z_{s_{p}}] - c_{1}(L_{p},  h_{p})=dd^{c}\log{|\langle a, \Gamma_{p} \rangle|_{h_{p}}}.\end{equation}By a basic feature of the variance, \begin{equation}\label{variancesum}
\mathrm{Var}\langle[Z_{s_{p}}], \phi \rangle= \mathrm{Var}\langle W_{s_{p}}, \phi \rangle .
\end{equation}
Therefore, in the light of (\ref{variancesum}) it is enough to estimate $\mathrm{Var}\langle W_{s_{p}}, \phi \rangle$.
 Employing  certain methods from \cite{SZ99} and \cite{SZ08} in our setting, we have the following theorem for codimension one.
\begin{thm}\label{cod1}
  Let \( (L_{p}, h_{p})_{p \geq 1} \), \( (X, \omega) \), and \( \sigma_{p} \) be as defined above. Assume that they satisfy the condition \textbf{(C2)}. If $s_{p}\in \mathcal{S}_{p} \subset H^{0}(X, L_{p})$, then for all $p\geq 1$ and any  $\phi \in\mathcal{D}^{n-1, n-1}(X)$, we have the following variance estimate \begin{equation}\label{ddcform}
  	\mathrm{Var} \langle [\widehat{Z}_{s_{p}}], \phi \rangle  \leq \frac{(C_{p})^{2/\alpha}}{A^{2}_{p}} (\int_{X}{|dd^{c} \phi|})^{2}.
  \end{equation}
Moreover, using (\ref{hercomp}), one gets \begin{equation*}\mathrm{Var} \langle [\widehat{Z}_{s_{p}}], \phi \rangle \leq \frac{(C_{p})^{2/\alpha}}{A^{2}_{p}} (B_{\phi} \mathrm{Vol}(X))^{2}, \end{equation*}where $B_{\phi}= b \|\phi\|_{\mathscr{C}^{2}}$ is a constant depending on the form $\phi$.
\end{thm}

\begin{proof}
First, by (\ref{Wss}), we have
\begin{equation}\label{cor}
	\mathbb{E}\langle W_{s_{p}}, \, \phi \rangle ^{2}= \int_{\mathcal{S}_{p}}\int_{X}\int_{X}{\log{|\langle a, \Gamma_{p}(x) \rangle|_{h_{p}}}\log{|\langle a, \Gamma_{p}(y) \rangle|_{h_{p}}}dd^{c}\phi(x)dd^{c}\phi(y)d\sigma_{p}(s)}.
\end{equation}
Now, writing
$$
|\Gamma_{p}(x)|_{h_{p}} := \Big(\sum_{j=1}^{d_{p}}{|S^{p}_{j}(x)|^{2}_{h_{p}}}\Big)^{1/2} = \sqrt{K_{\mathcal{S}_{p}}(x)}
$$
allows us to express \( \Gamma_{p}(x) \) as \( \Gamma_{p}(x) = |\Gamma_{p}(x)|_{h_{p}} u_{p}(x) \), where \( u_{p}(x) \) is a unit vector in the direction of \( \Gamma_{p}(x) \) such that \( |u_{p}(x)|_{h_{p}} = 1 \). Substituting \( \Gamma_{p}(x) = |\Gamma_{p}(x)|_{h_{p}} u_{p}(x) \) into the integrand in (\ref{cor}) breaks it into four terms:
\begin{align}
\log{|\Gamma_{p}(x)|_{h_{p}}}\log{|\Gamma_{p}(y)|_{h_{p}}}+ \log{|\Gamma_{p}(x)|_{h_{p}}}\log{| \langle a, u_{p}(y)\rangle|_{h_{p}}} +  \log{|\Gamma_{p}(y)|_{h_{p}}}\log{| \langle a, u_{p}(x)\rangle|_{h_{p}}}
 \label{forones1} \\ \label{forones} +\log{| \langle a, u_{p}(x)\rangle|_{h_{p}}}\log{| \langle a, u_{p}(y)\rangle|_{h_{p}}}. \nonumber
\end{align}

Before continuing with the variance estimate, we will see that \( \mathbb{E}\langle W_{s_{p}}, \, \phi \rangle \) is bounded. To do this, we prove an auxiliary inequality to begin with. First, by (\ref{hercomp}), we have
\begin{equation} \label{jens}
  \Big| \int_{X}{\log{|\Gamma_{p}(x)|_{h_{p}}} \,dd^{c}\phi(x)} \Big| \leq
	\int_{X}{\big|\log{|\Gamma_{p}(x)|_{h_{p}}}\big| \, \big|dd^{c}\phi(x) \big|} \leq \frac{b \, \| \phi \|_{\mathscr{C}^{2}}}{2} \int_{X}{\big|\log{K_{\mathcal{S}_{p}}(x)}\big| \,\omega^{n}} < \infty,
\end{equation}
since \( \log{K_{\mathcal{S}_{p}}(x)} \in L^{1}(X, \omega^{n}) \) for every \( p \geq 1 \). Now, it is evident that
\begin{align} \label{3..6}
	\big{|} \mathbb{E}\langle W_{s_{p}}, \phi \rangle \big{|} &= \Big| \int_{\mathcal{S}_{p}} \int_{X}{\Big(\log{|\Gamma_{p}(x)|_{h_{p}}} + \log{| \langle a, u_{p}(x)\rangle|_{h_{p}}}\Big)} \, dd^{c}\phi(x) \, d\sigma_{p}(s_{p}) \Big| \\ \nonumber 
	&\leq \int_{\mathcal{S}_{p}} \int_{X}{\big|\log{|\Gamma_{p}(x)|_{h_{p}}}\big| \, \big| dd^{c}\phi(x) \big| \, d\sigma_{p}(s_{p})} \\ \nonumber
	&\quad + \int_{\mathcal{S}_{p}} \int_{X}{\big| \log{| \langle a, u_{p}(x)\rangle|_{h_{p}}}\big|} \, \big| dd^{c}\phi(x) \big| \, d\sigma_{p}(s_{p}). \nonumber
\end{align}
 The first integral has an upper bound by (\ref{jens}), given that \(\sigma_{p}\) is a probability measure on \(H^{0}(X, L_{p})\). For the second double integral, we identify \(\mathcal{S}_{p}\simeq\mathbb{C}^{d_{p}}\). Using the moment condition \textbf{(C2)}, H\"{o}lder's inequality, and (\ref{hercomp}), we have
 \begin{equation} \label{2010}
 	\int_{X}\int_{\mathbb{C}^{d_{p}}}{\big{|}\log{| \langle a, \rho_{p}(x)\rangle|}\big{|}} \, d\sigma_{p}(a) \, \big| dd^{c}\phi(x) \big| \leq (C_{p})^{\frac{1}{\alpha}} \, b \, \| \phi \|_{\mathscr{C}^{2}} \, \mathrm{Vol}(X),
 	\end{equation}
 where
 $$\rho_{p}(x) = \left( \frac{f_{1}(x)}{\sqrt{\sum_{j=1}^{d_{p}}{|f_{j}(x)|^{2}}}}, \ldots, \frac{f_{d_{p}}(x)}{\sqrt{\sum_{j=1}^{d_{p}}{|f_{j}(x)|^{2}}}} \right). $$
 Applying Fubini-Tonelli’s theorem, we can write
 \begin{equation}\label{seci}
 	\int_{\mathcal{S}_{p}}\int_{X}{\big| \log{| \langle a, u_{p}(x)\rangle|_{h_{p}}} \big|} \, \big| dd^{c}\phi(x) \big| \, d\sigma_{p}(s_{p}) = \int_{X}\int_{\mathbb{C}^{d_{p}}}{\big| \log{| \langle a, \rho_{p}(x)\rangle|} \big|} \, d\sigma_{p}(a) \, \big| dd^{c}\phi(x) \big|,
 \end{equation}
 and thus we obtain a bound for the second integral.

 We now return to estimating the variance of \( W_{s_{p}} \). Expanding the term \( \langle \mathbb{E}[W_{s_{p}}], \, \phi \rangle^{2} \) using the expression for the expected distribution, we have
 \begin{equation*}
 	\big( \mathbb{E}\langle W_{s_{p}}, \phi \rangle \big)^{2} = J_{1} + 2 J_{2} + J_{3},
 \end{equation*}
 where
 \begin{align}
 	J_{1} &= \left( \int_{\mathcal{S}_{p}} \int_{X}{\log{|\Gamma_{p}(x)|_{h_{p}}} \, dd^{c}\phi(x) \, d\sigma_{p}(s_{p})} \right)^{2}, \label{J1} \\
 	J_{2} &= \left( \int_{\mathcal{S}_{p}} \int_{X}{\log{|\Gamma_{p}(x)|_{h_{p}}} \, dd^{c}\phi(x) \, d\sigma_{p}(s_{p})} \right) \times \left( \int_{\mathcal{S}_{p}} \int_{X}{\log{| \langle a, u_{p}(x)\rangle|_{h_{p}}} \, dd^{c}\phi(x) \, d\sigma_{p}(s_{p})} \right), \label{J2} \\
 	J_{3} &= \left( \int_{\mathcal{S}_{p}} \int_{X}{\log{| \langle a, u_{p}(x)\rangle|_{h_{p}}} \, dd^{c}\phi(x) \, d\sigma_{p}(s_{p})} \right)^{2}. \label{J3}
 \end{align}
 Since \( \mathbb{E}\langle W_{s_{p}}, \phi \rangle \) is bounded, it follows that \( J_{1}, J_{2} \), and \( J_{3} \) are all finite.

 From the expressions (\ref{forones1}), we write \( \mathbb{E}\langle W_{s_{p}}, \phi \rangle^{2} = B_{1} + 2B_{2} + B_{3} \), where
 \begin{align}
 	B_{1} &= \int_{\mathcal{S}_{p}} \int_{X} \int_{X}{\log{| \Gamma_{p}(x)|_{h_{p}}} \, \log{| \Gamma_{p}(y)|_{h_{p}}} \, dd^{c}\phi(x) \, dd^{c}\phi(y) \, d\sigma_{p}(s_{p})}, \label{B1} \\
 	B_{2} &= \int_{\mathcal{S}_{p}} \int_{X} \int_{X}{\log{|\Gamma_{p}(x)|_{h_{p}}} \, \log{| \langle a, u_{p}(y)\rangle|_{h_{p}}} \, dd^{c}\phi(x) \, dd^{c}\phi(y) \, d\sigma_{p}(s_{p})}, \label{B2} \\
 	B_{3} &= \int_{X} \int_{X}{dd^{c}\phi(y) \, dd^{c}\phi(x)} \int_{\mathbb{C}^{d_{p}}}{\log{|\langle a, \rho_{p}(x) \rangle|} \, \log{|\langle a, \rho_{p}(y) \rangle|} \, d\sigma_{p}(a)}. \label{B3}
 \end{align}
 From (\ref{jens}), the moment assumption \textbf{(C2)}, and Fubini-Tonelli’s theorem, we observe that \( B_{1}, B_{2} \), and \( B_{3} \) are all finite, and also that \( B_{1} = J_{1} \) and \( B_{2} = J_{2} \). Thus, the remaining terms are \( J_{3} \) and \( B_{3} \), and we find
 \begin{equation}\label{varip}
 	\mathrm{Var}\langle W_{s_{p}}, \phi \rangle = B_{3} - J_{3}.
 \end{equation}
 It suffices to estimate \( B_{3} \) from above to complete the variance estimation. Using Tonelli's theorem and H\"{o}lder's inequality with \( \frac{1}{\alpha} + \frac{1}{\beta} = 1 \), where \( \alpha \geq 2 \) is the constant satisfying the moment condition \textbf{(C2)}, we have
  \begin{align}
 	|B_{3}| &\leq \int_{X} \int_{X}{|dd^{c}\phi(y)| \, |dd^{c}\phi(x)|} \int_{\mathbb{C}^{d_{p}}}{\big| \log{|\langle a, \rho_{p}(x) \rangle|} \big| \, \big| \log{|\langle a, \rho_{p}(y) \rangle|} \big| \, d\sigma_{p}(a)} \notag \\
 	&\leq \int_{X} \int_{X}{|dd^{c}\phi(y)| \, |dd^{c}\phi(x)|} \left\{ \int_{\mathbb{C}^{d_{p}}}{\big| \log{|\langle a, \rho_{p}(x) \rangle|} \big|^{\alpha} d\sigma_{p}(a)} \right\}^{\frac{1}{\alpha}} \notag  \left\{ \int_{\mathbb{C}^{d_{p}}}{\big| \log{|\langle a, \rho_{p}(y) \rangle|} \big|^{\beta} d\sigma_{p}(a)} \right\}^{\frac{1}{\beta}} \notag \\
 	&\leq \int_{X} \int_{X}{|dd^{c}\phi(y)| \, |dd^{c}\phi(x)| \, (C_{p})^{\frac{1}{\alpha}}} \left\{ \int_{\mathbb{C}^{d_{p}}}{\big| \log{|\langle a, \rho_{p}(y) \rangle|} \big|^{\beta} d\sigma_{p}(a)} \right\}^{\frac{1}{\beta}}. \notag
 \end{align}
 Applying H\"{o}lder's inequality again to the innermost integral in the last line (since \( \alpha \geq 2 \geq \beta \) allows us to do so), we get
 \begin{equation}\label{lasth}
 	B_{3} \leq (C_{p})^{\frac{2}{\alpha}} \int_{X}\int_{X}{|dd^{c}\phi(y)| \, |dd^{c}\phi(x)|} \leq (C_{p})^{\frac{2}{\alpha}}\,\, (\int_{X}{|dd^{c} \phi|})^{2}.
 \end{equation}
 Consequently, applying the total variation inequality (\ref{hercomp}) for \( dd^{c} \) twice in (\ref{lasth}), we obtain the following inequality:
 \begin{equation}\label{A'}
 	B_{3} \leq (C_{p})^{2/\alpha} \, b^{2} \, \| \phi \|^{2}_{\mathscr{C}^{2}} \, \mathrm{Vol}(X)^{2},
 \end{equation}
 which, after dividing by $A^{2}_{p}$ and defining \( B_{\phi} := b\, \| \phi \|_{\mathscr{C}^{2}} \), provides the variance estimate as intended.
\end{proof}

For $\phi \in \mathcal{D}^{n-1, n-1}(X)$, from the arguments in the proof of Theorem \ref{cod1}, we have
\begin{equation*}\label{drr}
\mathbb{E}\langle [\widehat{Z}_{s_{p}}], \phi \rangle=\frac{1}{A_{p}} \int_{X} {c_{1}(L_{p}, h_{p}) \wedge \phi} + \int_{X}{\log{|K_{\mathcal{S}_{p}}(x)|_{h_{p}} dd^{c}\phi(x)}} + \int_{X}\int_{\mathcal{S}_{p}}{\log{|\langle a, u_{p}(x)\rangle|_{h_{p}} d\sigma_{p}(s_{p})dd^{c}\phi(x)}}.
\end{equation*}

Notice that even though $dd^{c} \phi$ is a signed measure, it is still possible to apply Fubini-Tonelli's theorem to interchange the integrals using the Jordan decomposition of $dd^{c} \phi$. Since $\log{\big|\langle a, u_{p}(x)\rangle\big|_{h_{p}}}= \log{\frac{|s_{p}(x)|_{h_{p}}}{\sqrt{K_{\mathcal{S}_{p}}(x)}}}$, it follows from this last expression that

\begin{equation}\label{excur}
	\mathbb{E}[\widehat{Z}_{s_{p}}]=\frac{1}{A_{p}}c_{1}(L_{p}, h_{p}) + \frac{1}{2A_{p}}dd^{c}\log K_{\mathcal{S}_{p}}(x)+\frac{1}{A_{p}}dd^{c} \big(\int_{s_{p} \in \mathcal{S}_{p}}{\log\big|\langle a, u_{p}(x)\rangle\big|_{h_{p}} d\sigma_{p}(s_{p})} \big),
\end{equation}which shows that $\mathbb{E}\langle [\widehat{Z}_{s_{p}}], \phi \rangle= \langle \mathbb{E}[\widehat{Z}_{s_{p}}], \phi \rangle$, that is, the expected value of a smooth linear statistic can be seen as a current, also referred to as the expected current of integration, which is a positive, closed $(1, 1)$-current.

Let $(L, h)$ be a $\mathscr{C}^2$-Hermitian line bundle over $X$ and $\mathcal{S} \subset H^{0}(X, L)$. Suppose that the base locus $\operatorname{Bs}(\mathcal{S})$ is empty. We observe that, for $s\in \mathcal{S}$, $\mathbb{E}[Z_{s}]$ is of Bedford-Taylor class (and so the wedge product of these currents in our setting is well-defined). To see this, given an orthonormal basis $\{S_{j}\}^{d}_{j=1}$ of $\mathcal{S}$ and $s\in \mathcal{S} \simeq \mathbb{C}^d$ with $s=\sum_{j=1}^{d}{a_{j} S_{j}}$, we look at the local picture, i.e., holomorphic frame representation $S_{j}=f_{j}\,e$ on some open set $U$, where each $f_{j}$ is a holomorphic function on $U$ for $j=1, \ldots, d$, and $e$ is a non-vanishing holomorphic section with $|e|_{h}=e^{-\varphi}$ for some $\mathscr{C}^{2}$ weight function $\varphi$ on $U$. Then $s=\langle a, f \rangle \,e$, where $f=(f_{1}, \ldots, f_{d})$. Let us write $F(x)=\int_{a\in \mathbb{C}^d}{\log| \langle a, f(x)\rangle| d\sigma(a)}$. $F$ is (locally uniformly) bounded on $U$ since the Bergman kernel is locally uniformly bounded (since the base locus is empty), $\varphi$ is $\mathscr{C}^2$ and \textbf{(C2)} holds, so, by Fubini-Tonelli's theorem we first have the following
\begin{equation}\label{lpf}
\int_{s\in \mathcal{S}}{\langle [Z_{s}], \phi \rangle d\sigma(s)}=\int_{a\in \mathbb{C}^d}\int_{X}{\log{|\langle a, f(x) \rangle|} \,dd^c\phi \, d\sigma(a)}= \langle \, dd^c \, \Big(\int_{a\in \mathbb{C}^d}{\log| \langle a, f(x)\rangle| d\sigma(a)} \Big), \phi \rangle.
\end{equation} Since expected current is positive (and closed), it follows from (\ref{lpf}) and basic pluripotential theory (see, for example, \cite[Theorem 4.15]{BrTr07}), the non-negativity of $dd^c F$ implies that there is a plurisubharmonic function $G$ on $U$ such that $F=G$ almost everywhere, so $\mathbb{E}[Z_{s}]=dd^c G=dd^c F$.

We present some lemmata, one of them is the following concerning cohomology classes of integration currents that will be instrumental in the sequel. This lemma has been previously proven as part of \cite[Theorem 3.1]{Shif}, the relation $(35)$ there, and for the sake of the reader, we give its proof here.

\begin{lem}\label{coh_lemma}
	Let $(X,\omega)$ be a compact K\"{a}hler manifold with the fixed K\"{a}hler form $\omega.$ If $s_{j} \in \mathcal{S}_{j} \subset H^{0}(X, L_{j}), \,\,j=1, \ldots, k$, are smooth and intersect transversally, then for $\Sigma^{k}=(s_1, \ldots, s_k)$, we have
	\begin{equation}\label{cohol}
		\left \langle \big[Z_{\Sigma^{k}}\big], \omega^{n-k} \right \rangle = \int_{X}c_{1}(L_{1}, h_{1})\wedge \cdots \wedge c_{1}(L_{k}, h_{k}) \wedge \omega^{n-k}.
	\end{equation}
	\begin{proof}
		For $k=1,$ this is just a consequence of Poincar\'{e}-Lelong formula and the fact that $\omega$ is a closed form. Indeed,
		\begin{equation*}
			\left \langle \big[Z_{s_{1}}\big], \omega^{n-1} \right \rangle = \int_{X}c_{1}(L_{1}, h_{1})\wedge \omega^{n-1}+\int_{X} dd^{c}\log |s_{1}|_{h_{1}}\wedge\omega^{n-1}=\int_{X}c_{1}(L_{1}, h_{1})\wedge \omega^{n-1}.
		\end{equation*}
		Let us now suppose that the assertion (\ref{cohol}) is true for $k-1$ sections $\Sigma^{k-1}=(s_{2},\ldots, s_{k}).$ Then by the induction hypothesis and the base  step of induction,  it yields that
		\begin{align*}
				 \left \langle \big[Z_{s_{1}}\cap Z_{\Sigma^{k-1}}\big], \omega^{n-k} \right \rangle &= \int_{Z_{s_{1}}}c_{1}(L_{2}, h_{2})\wedge \cdots \wedge c_{1}(L_{k}, h_{k}) \wedge \omega^{n-k}\\
			&=\int_{X} c_{1}(L_{1}, h_{1}) \wedge c_{1}(L_{2}, h_{2})\wedge \cdots \wedge c_{1}(L_{k}, h_{k}) \wedge\omega^{n-k} \end{align*}
	Since $\big[Z_{\Sigma^{k}}\big]= \big[Z_{s_{1}}\cap Z_{\Sigma^{k-1}}\big]$ by our assumption, we complete the proof.	\end{proof}
\end{lem}

\vspace{1mm}

\subsection{Random Systems}\label{sc3.1}
Let $(L_{p, j}, h_{p, j})$ be sequences of positive $\mathscr{C}^{2}$-Hermitian holomorphic line bundles for $j=1, \cdots, k$, over $X$. Fix some large $p_{0} \in \mathbb{N}$ so that the base locus is empty. Let $\sigma_{p}^{k}:=\sigma_{p, 1} \times \cdots \times \sigma_{p, k}$ be the product measure on $\mathcal{S}_{p, 1} \times \cdots \times \mathcal{S}_{p, k}$, where $S_{p, j}$ is a (finite) dimensional subspace of $H^{0}(X, L_{p, j})$ for $j=1, \ldots, k$. Consider the set $\mathcal{B}$ of systems $\Sigma^k_{p}=(s^{1}_{p}, \ldots, s^{k}_{p})$ with the property that $Z_{s^{j}_{p}}$ is smooth for $j=1, \ldots, k$. By the classical Bertini's theorem, \textbf{(C1)} and the product measure, we get $\sigma^{k}_{p}(\mathcal{B})=1$. Also, by Proposition \ref{bertiniseveral}, the set $\mathcal{A}$ of systems $\Sigma^{k}_{p}=(s^{1}_{p}, \cdots, s^{k}_{p})$ such that the individual zero sets $Z_{s^{j}_{p}}$ are in general position has full $\sigma^{k}_{p}$-measure. Therefore, we immediately have that $\sigma^k_{p}(\mathcal{A} \cap \mathcal{B})=1$, that is, for $\sigma^{k}_{p}$-almost all $(s^{1}_{p}, \ldots, s^{k}_{p}) \in \mathcal{S}_{p, 1} \times \ldots \times \mathcal{S}_{p, k}$, \,$Z_{\Sigma_{p}^{k}}$ is a smooth complete intersection of $Z_{s^{j}_{p}}$, i.e., it is a compact complex submanifold of codimension $k$. We shall continue to refer to such elements as \textit{typical} as in the Introduction. Proposition \ref{wdberti} gives that $[Z_{\Sigma^{k}_{p}}]:= [Z_{s^{1}_{p}}] \wedge \cdots \wedge [Z_{s^{k}_{p}}]$ is well-defined and is equal to the current of integration with multiplicities over the complete intersection $Z_{\Sigma^{k}_{p}}$ of pure dimension $n-k$. Let $\Sigma^{k}_{p}=(s^{1}_{p}, s^{2}_{p},\ldots , s^{k}_{p})$ be such a typical system of independent random holomorphic sections $s_{p}^{j}\in \mathcal{S}_{p, j} \subset H^{0}(X,L_{p, j})$ for $j=1,2, \cdots,k\,$ where $1\leq k \leq \textup{dim}_{\mathbb{C}}X=n$. Let $\phi \in \mathcal{D}^{n-k, n-k}(X)$. Then, by applying Lemma \ref{coh_lemma}, we get
\begin{equation}\label{uniboex}
\big|\big \langle \big[Z_{\Sigma_{p}^{k}}\big], \phi \big \rangle\big|=\Big|\int_{Z_{\Sigma_{p}^{k}}}\phi \, \Big|\leq \sup \left \| \phi \right \| \int_{Z_{\Sigma_{p}^{k}}}{\omega ^{n-k}} \leq \sup \left \| \phi \right \| \,\int_{X}c_{1}(L_{p, 1}, h_{p, 1})\wedge \cdots \wedge c_{1}(L_{p, k}, h_{p, k}) \wedge \omega^{n-k},\end{equation}
which means that $\big \langle \big[Z_{\Sigma_{p}^{k}}\big], \phi \big \rangle$ is bounded  for  almost all $\Sigma_{p}^{k} .$ Consequently,  $\mathbb{E}\langle [Z_{\Sigma_{p}^{k}}], \phi \rangle$ is well-defined.

The currents, such as $[Z_{s^1_{p}}] \wedge \mathbb{E}_{\sigma_{p, 2}}[Z_{s^2_{p}}]$, are also well-defined by \cite[Chapter 3, Corollary 4.11]{Dem12}, so \begin{equation} \label{restt}\langle [Z_{s^1_{p}}] \wedge \mathbb{E}_{\sigma_{p, 2}}[Z_{s^2_{p}}], \phi \rangle =\int_{Z_{s^1_{p}}}{\mathbb{E}_{\sigma_{p, 2}}[Z_{s^2_{p}}] \wedge \phi}.\end{equation}

\begin{thm} \label{indpw}

Let $(X,\omega)$ be a compact complex manifold of dimension $n$, and fix $1\le k\le n$.
For each $p\geq 1$, let $(L_{p,j},h_{p,j})\to X$ ($j=1,\dots,k$) be sequences of $\mathscr{C}^{2}$-Hermitian holomorphic line bundles over $X$, and let $\mathcal{S}_{p,j}\subset H^0(X,L_{p,j})$ be subspaces endowed with the probability measures $\sigma^{j}_{p}$ satisfying \textbf{(C2)} and \textbf{(C1)}. Let $\phi \in \mathcal{D}^{n-k, n-k}(X)$. Assume that for every $j=1,\ldots,k$,
the linear system $\mathcal{S}_{p,j}$ is base-point free, i.e. $\operatorname{Bs}(\mathcal{S}_{p,j})=\varnothing$, for large values of $p$. Then, for large enough values of $p\in \mathbb{N}$ and independent random systems $\Sigma^k_{p}=(s^{1}_{p}, \ldots, s^{k}_{p})\in \mathcal{S}_{p,1} \times \ldots \times \mathcal{S}_{p, k}$, the expected simultaneous zero current satisfies
$$
\mathbb{E}_{\sigma_{p,1}\times\cdots\times\sigma_{p,k}}\!\big\langle [Z_{\Sigma^{k}_{p}}], \phi \big\rangle
\;=\; \,\!\big\langle \bigwedge_{j=1}^{k} \mathbb{E}_{\sigma_{p,j}}[Z_{\,s^{j}_{p}}] ,\phi \big\rangle:= \langle \mathbb{E}_{\sigma_{p,1}\times\cdots\times\sigma_{p,k}}\big[Z_{\Sigma^{k}_{p}}\big], \phi \rangle,
$$where, as was given in (\ref{excur}), \begin{equation*}
	\mathbb{E}_{\sigma_{p, j}}[Z_{s^{j}_{p}}]=c_{1}(L_{p, j}, h_{p, j}) + \frac{1}{2}dd^{c}\log K_{\mathcal{S}_{p, j}}(x)+dd^{c} \big(\int_{s^{j}_{p} \in \mathcal{S}_{p, j}}{\log\big|\langle a, u^{j}_{p}(x)\rangle\big|_{h_{p, j}} d\sigma_{p, j}(s^{j}_{p})} \big),
\end{equation*} As a result, $\mathbb{E}_{\sigma_{p,1}\times\cdots\times\sigma_{p,k}}\big[Z_{\Sigma^k_{p}}\big]$ is a positive, closed $(k, k)$-current.

\end{thm}

\begin{proof}

	We induct on the codimension $k.$ The base step $k=1$ follows from (\ref{excur}). Assume that the formula holds for $k-1$ sections, that is for any typical system $\Sigma^{k-1}_{p}=(s_{p}^{2},\dots,s_{p}^{k}) \in \mathcal{S}_{p, 2} \times \ldots \times \mathcal{S}_{p, k}$. Let $\phi \in \mathcal{D}^{n-k, n-k}(X)$. Take a typical system $\Sigma^{k}_{p}=(s_{p}^{1}, \dots, s_{p}^{k}) \in \mathcal{S}_{p, 1} \times \ldots \times \mathcal{S}_{p, k}$. Write $\sigma^{k}_{p}:= \sigma_{p, 1} \times \ldots \times \sigma_{p, k}$ and $\sigma^{k-1}_{p}:= \sigma_{p, 2} \times \ldots \times \sigma_{p, k}$as above for short. It will be enough to show that $ \mathbb{E}_{\sigma^{k}_{p}}\big[Z_{\Sigma^{k}_{p}}\big]=\mathbb{E}_{\sigma_{p, 1}}[Z_{s^1_{p}}] \wedge  \mathbb{E}_{\sigma^{k-1}_{p}}\big[Z_{\Sigma^{k-1}_{p}}\big]$. By using the wedge product decomposition $\langle [Z_{\Sigma^{k}_{p}}], \phi \rangle =\langle [Z_{s^{1}_{p}}] \wedge [Z_{\Sigma^{k-1}_{p}}], \phi \rangle$, definition of a current of integration, the Fubini-Tonelli theorem and the induction hypothesis, we have
	\begin{align*}
		\int_{\mathcal{S}_{p, 2} \times \cdots \times \mathcal{S}_{p, k}}\big \langle \big[Z_{\Sigma_{p}^{k}}\big], \phi \big \rangle d\sigma^{k-1}_{p}(\Sigma_{p}^{k-1})&=\int_{\mathcal{S}_{p, 2} \times \cdots \times \mathcal{S}_{p, k}}\left \langle [Z_{s_{p}^{1}}] \wedge [Z_{\Sigma_{p}^{k-1}}], \phi \right \rangle d\sigma^{k-1}_{p}(\Sigma_{p}^{k-1})\\
		&= \int_{\mathcal{S}_{p, 2} \times \cdots \times \mathcal{S}_{p, k}} \big(\int_{Z_{s_{p}^{1}}} { [Z_{\Sigma_{p}^{k-1}}]} \,\wedge \phi \big) \, d\sigma_{p}^{k-1}(\Sigma_{p}^{k-1})   \\
		&=\int_{Z_{s^{1}_{p}}}{\mathbb{E}_{\sigma^{k-1}_{p}}\big[Z_{\Sigma^{k-1}_{p}}\big] \wedge \phi}\\
        &=\big \langle [Z_{s_{p}^{1}}] \wedge \mathbb{E}_{\sigma^{k-1}_{p}}[Z_{\Sigma_{p}^{k-1}}], \phi \big \rangle, \nonumber \\
	\end{align*}where, in the last equality, we have used (\ref{restt}) with $\mathbb{E}_{\sigma_{p, 2} \times \cdots \times \sigma_{p, k}}[Z_{\Sigma^{k-1}_{p}}]$. Finally, integrating over all $s_{p}^{1}$ with respect to the measure $\sigma_{p, 1}$ in the last expression and using the initial case of induction one more time give what is desired.

\end{proof} Now we go on with the proof of the variance estimate in higher codimensions. We adapt the methods in \cite[Theorem 3.1]{Shif} into our setting.

\begin{thm} \label{genvar}
Let $L_{p, 1}, \ldots, L_{p, k}, \,\,k\in \{1, \ldots, n\}$, be sequences of $\mathscr{C}^{2}$-Hermitian holomorphic line bundles on a compact K\"{a}hler manifold $(X, \omega)$. Assume that the subspaces $\mathcal{S}_{p, j} \subset H^{0}(X, L_{p, j}),\,j=1,\ldots, k$, equipped with the probability measures $\sigma_{p, j}$ satisfy \textbf{(C1)-(C2)}, and that for every $j=1,\ldots,k$,
the linear system $\mathcal{S}_{p,j}$ is base-point free, i.e. $\operatorname{Bs}(\mathcal{S}_{p,j})=\varnothing$ for large values of $p$. Then for $\Sigma^{k}_{p}=(s_{p}^{1}, \ldots, s_{p}^{k})\in \mathcal{S}_{p, 1} \times \ldots \times \mathcal{S}_{p, k}$ with random sections $s_{p}^{1}, \ldots, s_{p}^{k}$ selected independently with respect to the product probability measure $\sigma_{p, 1} \times \cdots \times \sigma_{p, k}$, we have

\begin{equation} \label{varra}\mathrm{Var} \langle [Z_{\Sigma_{p}^{k}}], \phi \rangle \leq (C_{p})^{2/\alpha}\,(B_{\phi})^{2}(\int_{X}{\omega^{n-k+1} \wedge \sum_{\beta=1}^{k}{\Big[ \prod_{1\leq j \leq k, j \neq \beta}{c_{1}(L_{p, j}, h_{p, j})}\Big]}\,})^{2}\end{equation}

\end{thm}

\begin{proof}

Pick a typical system of $k$ independent random holomorphic sections
$\Sigma_p^{k}=(s_p^{1},\ldots,s_p^{k})\in \mathcal{S}_{p,1}\times\cdots\times \mathcal{S}_{p,k}$ and write
$\Sigma_p^{k}=(\Sigma_p^{k-1},s_p^{k})$ with $\Sigma_p^{k-1}=(s_p^{1},\ldots,s_p^{k-1})$.
For typical choices one has
$[Z_{\Sigma_p^{k}}]=[Z_{\Sigma_p^{k-1}}]\wedge [Z_{s_p^{k}}]$.
Let $\phi\in \mathcal{D}^{n-k,n-k}(X)$ be a test form.

By Theorem \ref{indpw}, we have
$$
\mathbb{E}[Z_{\Sigma_p^{k}}]=\mathbb{E}[Z_{\Sigma_p^{k-1}}]\wedge \mathbb{E}[Z_{s_p^{k}}].
$$
Hence
\begin{equation}\label{varidef}
\begin{aligned}
\mathrm{Var}\big(\langle [Z_{\Sigma_p^{k}}],\phi\rangle\big)
&=\mathbb{E}\,\langle [Z_{\Sigma_p^{k}}],\phi\rangle^{2}
-\big(\mathbb{E}\,\langle [Z_{\Sigma_p^{k}}],\phi\rangle\big)^{2} \\
&=\mathbb{E}\,\langle [Z_{\Sigma_p^{k-1}}]\wedge [Z_{s_p^{k}}],\phi\rangle^{2}
-\Big(\big\langle \mathbb{E}[Z_{\Sigma_p^{k-1}}]\wedge \mathbb{E}[Z_{s_p^{k}}],\phi\big\rangle\Big)^{2}.
\end{aligned}
\end{equation}
We write
$$
\langle [Z_{\Sigma_p^{k-1}}]\wedge [Z_{s_p^{k}}],\phi\rangle^{2}
-\big\langle \mathbb{E}[Z_{\Sigma_p^{k-1}}]\wedge \mathbb{E}[Z_{s_p^{k}}],\phi\big\rangle^{2}
= I_1+ \ldots +I_k,
$$
where

\begin{align}
I_k(\Sigma_p^{k-1},s^{k}_{p}) &:=\langle [Z_{\Sigma_p^{k-1}}]\wedge [Z_{s_p^{k}}],\phi\rangle^{2}
-\langle [Z_{\Sigma_p^{k-1}}]\wedge \mathbb{E}[Z_{s_p^{k}}],\phi\rangle^{2}, \\ \nonumber
I_{k-1}(\Sigma_p^{k-2},s^{k-1}_{p}) &:=\langle [Z_{\Sigma_p^{k-2}}]\wedge [Z_{s^{k-1}_{p}}] \wedge \mathbb{E}[Z_{s_p^{k}}],\phi\rangle^{2}
-\big(\langle [Z_{\Sigma_p^{k-2}}]\wedge \mathbb{E}[Z_{s^{k-1}_{p}}] \wedge \mathbb{E}[Z_{s_p^{k}}],\phi\rangle\big)^{2}, \\ \nonumber
\vdots \\ \nonumber
I_{1}(s^{1}_{p}) &:= \langle [Z_{s_p^{1}}]\wedge \mathbb{E}[Z_{s_p^{2}}] \wedge \cdots \wedge \mathbb{E}[Z_{s^{k}_{p}}],\phi\rangle^{2}
-\langle \mathbb{E}[Z_{s_p^{1}}]\wedge \mathbb{E}[Z_{s^{2}_{p}}] \wedge \cdots \wedge \mathbb{E}[Z_{s_p^{k}}],\phi\rangle^{2}. \nonumber
\end{align}
These quantities are well-defined for typical choices (see the proof of Theorem \ref{indpw}).
Set $\sigma_p^{k-1}:=\sigma_{p,1}\times\cdots\times \sigma_{p,k-1}$.
By taking expectations in the above identity, we have
\begin{equation}\label{sumvar}
\mathrm{Var}\big(\langle [Z_{\Sigma_p^{k}}],\phi\rangle\big)=\mathbb{E}[I_1]+\mathbb{E}[I_2] + \cdots + \mathbb{E}[I_{k}].
\end{equation}

\noindent\textbf{Estimation of $\mathbb{E}[I_k]$:}

Let us fix $\Sigma_p^{k-1}$ and put
$$
T:=[Z_{\Sigma_p^{k-1}}]=[Z_{s_p^{1}}]\wedge\cdots\wedge [Z_{s_p^{k-1}}],
$$
a positive closed current of bidegree $(k-1,k-1)$.
Then
\begin{equation}\label{condvarIk}
\int_{\mathcal{S}_{p,k}} I_k(\Sigma_p^{k-1},s_p^{k})\,d\sigma_{p,k}(s_p^{k})
=
\mathrm{Var}_{\sigma_{p, k}}\Big(\big\langle [Z_{s_p^{k}}]\wedge T,\phi\big\rangle\Big).
\end{equation}

\smallskip

Define the random variable on $(\mathcal{S}_{p,k},\sigma_{p,k})$ by
$$
X(s_p^{k})
:=\big\langle [Z_{s_p^{k}}]\wedge T ,\phi\big\rangle .
$$
Set
$$
\mu:=T \wedge dd^c\phi.
$$

\smallskip

We do not apply Theorem~\ref{cod1} directly (since $T$ need not be smooth), instead,
we follow the idea of the proof of Theorem~\ref{cod1} with the signed measure
$\mu$ in place of $dd^c\phi$.
Since $T$ is closed and independent of the randomness in $s_p^k$, the Poincar\'e--Lelong formula (\ref{lelpoi}) and Demailly's theory (see \cite[Chapter III, Section 4]{Dem12}) show that $X$ is equal to a deterministic constant plus
$\int_X \log|s_p^k|_{h_{p,k}}\, \mu$.
Decomposing $\log|s_p^k|_{h_{p,k}}$ as in the proof of Theorem~3.1 and using the same arguments with the moment condition
\textbf{(C2)}, we obtain
\begin{equation}\label{cod1psi-ell}
\mathrm{Var}_{\sigma_{p,k}}\!(X)
=
\mathrm{Var}_{\sigma_{p,k}}\!\left(\big\langle [Z_{s_p^{k}}]\wedge T, \phi\big\rangle\right)
\le (C_p)^{2/\alpha}\left(\int_X |\mu|\right)^2
=
(C_p)^{2/\alpha}\left(\int_X \big|T \wedge dd^c\phi\big|\right)^2,
\end{equation} which yields, by Lemma \ref{coh_lemma} and (\ref{hercomp}),
\begin{equation}\label{VarXl-bound}
\mathrm{Var}_{\sigma_{p,k}}\!\left(\big\langle [Z_{s_p^{k}}]\wedge T,\phi\big\rangle\right)
\le (C_p)^{2/\alpha}(B_\phi)^{2}
\left(
\int_X \omega^{n-k+1}\wedge \bigwedge_{j=1}^{k-1}c_1(L_{p,j},h_{p,j})
\right)^{2}.
\end{equation}

Together with \eqref{condvarIk}, (\ref{VarXl-bound}) and the fact that $\sigma_p^{k-1}$ is a probability measure, this gives
the desired estimate for $\mathbb{E}[I_k]$:
\begin{equation}\label{eik}
\begin{aligned}
\mathbb{E}[I_k]
&=\int_{\mathcal{S}_{p,1}\times\cdots\times\mathcal{S}_{p,k-1}}
\int_{\mathcal{S}_{p,k}} I_k(\Sigma_p^{k-1},s_p^{k})\,d\sigma_{p,k}(s_p^{k})\,d\sigma_p^{k-1}(\Sigma_p^{k-1}) \\
&\le (C_p)^{2/\alpha}(B_\phi)^{2}
\left(\int_X \omega^{n-k+1}\wedge \bigwedge_{j=1}^{k-1}c_1(L_{p,j},h_{p,j})\right)^{2}.
\end{aligned}
\end{equation}

\medskip

\vspace{5mm}

\noindent\textbf{Estimations of $\mathbb{E}[I_j], \,j=1, \ldots, k-1$:} All estimations can be done in the same way by using Theorem \ref{indpw}, so we will treat just one of them in detail. To this end, consider the term $I_{1}$. Let us write $\Sigma^{k-1}_{p}:= (s^{2}_{p}, \ldots, s^{k}_{p})$ only for this argument, the reader should note that, as observed above, this notation actually indicates that the sections decrease from the end, however, for the sake of brevity, we introduce this minor modification of the notation in this part of the proof and accordingly let $[Z_{\Sigma^{k-1}_{p}}]=[Z_{s^{2}_{p}}] \wedge \cdots \wedge [Z_{s_{p}^{k}}]$. In order to get the upper bound for $\mathbb{E}[I_{1}]$, first observe, by Theorem \ref{indpw},  \begin{equation} \label{werr} \mathbb{E}[I_{1}]= \mathbb{E}\big\langle \big([Z_{s_{p}^{1}}]- \mathbb{E} [Z_{s_{p}^{1}}]\big) \wedge \mathbb{E}[Z_{\Sigma_{p}^{k-1}}], \phi  \big\rangle^{2}.\end{equation}Also we see that \begin{align*}\big\langle \big([Z_{s_{p}^{1}}]- \mathbb{E} [Z_{s_{p}^{1}}]\big) \wedge \mathbb{E}[Z_{\Sigma_{p}^{k-1}}], \phi  \big\rangle^{2} &=\Big\{\int_{\mathcal{S}_{p, 2} \times \ldots \times \mathcal{S}_{p, k}}{\big\langle\big([Z_{s_{p}^{1}}]- \mathbb{E}[Z_{s_{p}^{1}}]\big) \wedge [Z_{\Sigma_{p}^{k-1}}], \phi \big\rangle d\sigma_{p}^{k-1}(\Sigma_{p}^{k-1})} \Big\}^{2}  \\
& \leq \int_{\mathcal{S}_{p, 2} \times \ldots \times \mathcal{S}_{p, k}}{\Big\{\big\langle([Z_{s_{p}^{1}}]- \mathbb{E}[Z_{s_{p}^{1}}]) \wedge [Z_{\Sigma_{p}^{k-1}}], \phi \big\rangle \Big\}^{2} d\sigma^{k-1}_{p}(\Sigma_{p}^{k-1})}, \end{align*}where, in the second line, we have used Cauchy-Schwarz inequality. We then get \begin{equation}\label{varsev}
         \mathbb{E}[I_{1}]\leq \int_{\mathcal{S}_{p, 2} \times \ldots \times \mathcal{S}_{p, k}}\int_{\mathcal{S}_{p, 1}}\Big\{\big\langle([Z_{s_{p}^{1}}]- \mathbb{E}[Z_{s_{p}^{1}}]) \wedge [
         Z_{\Sigma_{p}^{k-1}}], \phi \big\rangle\Big\}^{2}\,  d\sigma_{p, 1}(s_{p}^{1})\, d\sigma^{k-1}_{p}(\Sigma_{p}^{k-1}). \end{equation} The righthand side of (\ref{varsev}) takes the following form by using the same notation $T= [Z_{\Sigma^{k-1}_{p}}]$, \begin{equation}\label{varsev2}
\int_{\mathcal{S}_{p, 2} \times \ldots \times \mathcal{S}_{p, k}}{\mathrm{Var}_{\sigma_{p, 1}} \big(\big\langle [Z_{s_{p}^{1}}] \wedge T, \phi \big\rangle \big) d\sigma_{p}^{k-1}(\Sigma_{p}^{k-1})}.\end{equation} Proceeding in the same way as in the estimation of $\mathbb{E}[I_k]$, we arrive at the following expression:

\begin{equation}\label{ei1}
\begin{aligned}
\mathbb{E}[I_1]
&=\int_{\mathcal{S}_{p, 2}\times\cdots\times\mathcal{S}_{p, k}}
\int_{\mathcal{S}_{p,1}} I_1(s_p^{1})\,d\sigma_{p, 1}(s_p^{1})\,d\sigma_p^{k-1}(\Sigma_p^{k-1}) \\
&\le (C_p)^{2/\alpha}(B_\phi)^{2}
\Bigg(\int_X \omega^{n-k+1}\wedge \bigwedge_{j=2}^{k}c_1(L_{p, j},h_{p, j})\Bigg)^{2}.
\end{aligned}
\end{equation}

Hence, by applying the same method for each $I_{j}(\Sigma_p^{j-1},s^{j}_{p})$,\, $j=2, \ldots, k-1$, using (\ref{sumvar}) and the basic inequality $A_{1}^2 + \ldots A^{2}_{k} \leq (A_{1} + \ldots + A_{k})^{2}$, we have (\ref{varra}). 
\end{proof}

 \begin{proof}[Proof of Theorem \ref{th1}]
 Let us take all holomorphic line bundles as identical and $\mathcal{S}_{p, j}= H^{0}(X, L_{p})$ in Theorem \ref{genvar}. We immediately have

 \begin{equation} \label{spevar}
  \mathrm{Var} \langle [Z_{\Sigma_{p}^{k}}], \phi \rangle \leq (C_{p})^{2/\alpha}\,(B_{\phi})^{2}(n \int_{X}{\omega^{n-k+1} \wedge c_{1}(L_{p}, h_{p})^{k-1}})^{2}.
 \end{equation}By the diophantine approximation \textbf{(A)}, there exists $p_{1} \in \mathbb{N}$ such that
 \begin{equation*}\label{cur1}\frac{A_{p}}{2}\,\, \omega \leq c_{1}(L_{p}, h_{p}) \leq 2A_{p}\,\, \omega \end{equation*} for all $p \geq p_{1}$, which implies that \begin{equation}\label{voll}\int_{X}{\omega^{n-k+1}\wedge c_{1}(L_{p}, h_{p})^{k-1}}\leq (2A_{p})^{k-1} \,\mathrm{Vol}(X) \leq 2^{n-1} A^{k-1}_{p} \mathrm{Vol}(X)\end{equation} for all $p \geq \max\{p_0, p_1\}$. Finally, by using the normalization and writing $D_{n}:= (n \,2^{n-1})^{2}$, Theorem 1.1 follows.

 \end{proof}

\section{Equidistribution of Zeros of Random Sections}

In this section, we provide the proof for Theorem \ref{th2}. We divide the proof of it into three separate theorems. We begin with the asymptotic behavior of the expected zero distribution. Building on the variance estimates obtained earlier and the expected distribution, we prove that the zero sets possess the self-averaging property provided a summability condition is met.

Since the estimate (\ref{uniboex}) is easily seen to hold for continuous real-valued forms, following \cite{SZ99} (see also \cite{MM1}), we restrict our attention to a countable $\mathscr{C}^{0}$-dense family of smooth forms and work with a test form from this family.

\subsection{Expected Distribution of Zeros}
We will prove expected distribution for the case of codimension one, followed by a generalization of the proof to handle higher codimensions. While one could rephrase the results of this section using subspace notation as in the previous section, in codimension $k$, the subspaces necessarily must be taken to be the same. Since it just a matter of slight notation change, we keep the main formulation from the Introduction and prefer to work with $H^{0}(X, L_{p})$.

\begin{thm}\label{exdist1}
	 Let $(L_{p}, h_{p})_{p\geq 1}$,\, $(X, \omega)$ and $\sigma_{p}$  be as defined above. Assume that they  satisfy the conditions \textbf{(A)}, \textbf{(B)} and \textbf{(C2)}. If $\ \lim_{p\rightarrow \infty}\frac{C_{p}^{1/\alpha}}{A_{p}}=0,$ then
	\begin{equation}
		\frac{1}{A_{p}}\,\mathbb{E}[Z_{s_{p}}]\longrightarrow \omega
	\end{equation}
	in the weak* topology of currents as $p\rightarrow \infty.$
\end{thm}
\begin{proof}
	Let $\{S^{p}_{j}\}_{j = 1}^{d_{p}}$ be an orthonormal basis of $H^{0}(X, L_{p})$ and $s_{p}\in H^{0}(X, L_{p})$. Then $$s_{p}=\sum_{j=1}^{d_{p}}{a_{j}^{p}S^{p}_{j}}=\langle a, \Gamma_{p} \rangle,$$ where $\Gamma_{p}=(S_{1}^{p}, \ldots, S^{p}_{d_{p}})$, and $a=(a_{1}, \ldots, a_{d_{p}})\in \mathbb{C}^{d_{p}}$. Let $x\in X$,\, $U\subseteq X$ \,be an open neighborhood of $x$ and $e_{p}$ be a holomorphic frame of $L_{p}$ in $U$. Then locally $S^{p}_{j}=f_{j}e_{p}$, where $f_{j}$ are holomorphic functions in $U$\, and so, by writing $f=(f_{1}, \ldots, f_{d_{p}}),$ we have $$s_{p}=\sum_{j=1}^{d_{p}}{a_{j}^{p}f_{j}e_{p}}=\langle a, f \rangle e_{p}.$$ By Poincar\'{e}-Lelong formula (\ref{lelpoi}), on the neighborhood $U$, we have
	\begin{equation}\label{pl}
		[Z_{s_{p}}]=dd^{c}\log{|\langle a, f \rangle|}= dd^{c}\log|\langle a, \Gamma_{p}\rangle|_{h_{p}}+c_{1}(L_{p}, h_{p}).
	\end{equation}
	 Let us now fix $\phi \in \mathcal{D}^{n-1, n-1}(X),$  without loss of generality we may assume that $\operatorname{supp}(\phi)\subset U$ as the general case follows by covering $\operatorname{supp}(\phi)$ by $U_{\alpha}$'s and using the compatibility conditions. Using the  definition of expectation and (\ref{pl}), we have
	\begin{equation} \label{Ip}
		\frac{1}{A_{p}}\langle \mathbb{E}[Z_{s_{p}}], \phi \rangle=	\frac{1}{A_{p}}\langle c_{1}(L_{p}, h_{p}), \phi \rangle + 	\frac{1}{A_{p}}  \int_{H^{0}(X, L_{p})}\int_{X}{\log{| \langle a, \Gamma_{p}(x)\rangle|_{h_{p}}}}dd^{c}\phi(x)d\sigma_{p}(s_{p})
	\end{equation}
	Let us denote the second term above by $I_{p}$. Then, by exploiting the fact that $\Gamma_{p}(x)=|\Gamma_{p}(x)|_{h_{p}}u_{p}(x)$ (so that $|u_{p}|_{h_{p}}=1$), we get
	\begin{align}\label{cd1}
|I_p|
&\le \frac{1}{A_p}\int_{H^{0}(X,L_p)}\int_X
\big|\log |\Gamma_p(x)|_{h_p}\big|\,|dd^{c}\phi(x)|\,d\sigma_p(s_p) \notag\\
&\quad + \frac{1}{A_p}\int_{H^{0}(X,L_p)}\int_X
\big|\log |\langle a,u_p(x)\rangle|_{h_p}\big|\,|dd^{c}\phi(x)|\,d\sigma_p(s_p).
\end{align} Utilizing (\ref{hercomp}) and the simple observation that $\frac{1}{A_{p}}\log K_{p}(x) \rightarrow 0$ as $p \rightarrow \infty$ from (\ref{0.11}),
	\begin{equation}\label{1.term}
		\frac{1}{A_{p}}\int_{H^{0}(X, L_{p})}\int_{X}|\log{|\Gamma_{p}(x)|_{h_{p}}}|\,|dd^{c}\phi(x)|\,d\sigma_{p}(s) \leq 	\frac{B_{\phi}}{2A_{p}}\int_{X}|\log K_{p}(x)| \omega^{n}(x)\rightarrow 0
	\end{equation} as $p \rightarrow \infty.$ Also, using the identification  $H^{0}(X, L_{p})\simeq\mathbb{C}^{d_{p}}$,  the moment condition \textbf{(C2)} along with  H\"{o}lder's inequality and (\ref{hercomp}),  we get \begin{equation} \int_{X}\int_{\mathbb{C}^{d_{p}}}{\big{|}\log{| \langle a, \rho_{p}(x)\rangle|}\big{|}}d\sigma_{p}(a)dd^{c}\phi(x)\leq (C_{p})^{\frac{1}{\alpha}}B_{\phi}\,\mathrm{Vol}(X),  \end{equation} where $$\rho_{p}(x)=\Big(\frac{f_{1}(x)}{\sqrt{\sum_{j=1}^{d_{p}}{|f_{j}(x)|^{2}}}}, \ldots, \frac{f_{d_{p}}(x)}{\sqrt{\sum_{j=1}^{d_{p}}{|f_{j}(x)|^{2}}}}\Big).$$It follows from Fubini-Tonelli's theorem that \begin{equation*}
	\int_{H^{0}(X, L_{p})}\int_{X}{\big |\log{| \langle a, u_{p}(x)\rangle|_{h_{p}}}\big|}\big |dd^{c}\phi(x)\big|d\sigma_{p}(s_{p})=\int_{X}\int_{\mathbb{C}^{d_{p}}}{\big|\log{| \langle a, \rho_{p}(x)\rangle|}\big|}d\sigma_{p}(a)\big|dd^{c}\phi(x)\big|.
	\end{equation*}Consequently, employing the given hypothesis, we deduce that
	\begin{equation}\label{2.term}
		\frac{1}{A_{p}}  \int_{H^{0}(X, L_{p})}\int_{X}{\big|\log{| \langle a, u_{p}(x)\rangle|_{h_{p}}}\big|}\big|dd^{c}\phi(x)\big|d\sigma_{p}(s_{p})\leq \frac{C_{p}^{1/{\alpha}}}{A_{p}}\,B_{\phi}\,Vol(X) \rightarrow 0,  \end{equation}
		as $p\rightarrow \infty$. In turn, (\ref{1.term}) and (\ref{2.term}) imply that $I_{p}\rightarrow 0,$ as $p \rightarrow \infty.$ Finally, by using (\ref{cur}) we have that $	\frac{1}{A_{p}}\langle c_{1}(L_{p}, h_{p}), \phi \rangle \rightarrow \langle \omega, \phi \rangle ,$ thus concluding the proof.
\end{proof}

\begin{thm}\label{exdist}
	
	Under the assumptions \textbf{(A)}, \textbf{(B)}, \textbf{(C1)}-\textbf{(C2)}, and   \begin{equation}\label{poscon} \lim_{p\rightarrow \infty}{\frac{C_{p}^{1/\alpha}}{A_{p}}}=0, \end{equation} we have
	\begin{equation*}
		\mathbb{E}\big[\widehat{Z}_{\Sigma_{p}^{k}}\big]\longrightarrow \omega^k
	\end{equation*}
	in the weak* topology of currents as $p\rightarrow \infty.$
\end{thm}
\begin{proof}
	
	We follow the lines of the proof of Proposition 3.5 in \cite{CLMM}. There exists $c > 0$ such that for every $\phi \in \mathcal{D}^{n-k, n-k}(X)$, $k \in \{1, \dots, n\}$, and every real $(1,1)$-form $\theta$ on $X$, the following inequality holds
	\begin{equation}\label{form2}
		-c\|\phi\|_{\mathscr{C}^0} \|\theta\|_{\mathscr{C}^0} \omega^{n-k+1} \leq \phi \wedge \theta \leq c\|\phi\|_{\mathscr{C}^0} \|\theta\|_{\mathscr{C}^0} \omega^{n-k+1}.
	\end{equation}
	
	For $p > \max\{p_0, p_1\}$, define $\Theta_p := \mathbb{E}[Z_{s^1_p}]$. Using (\ref{excur}) and Theorem \ref{indpw}, we write
	$$
	\mathbb{E}[\widehat{Z}_{\Sigma^k_p}] = [\mathbb{E}[\widehat{Z}_{s^1_p}]]^k = \frac{\Theta_p^k}{A_p^k}.$$
	Now, define
	\begin{equation}\label{algmanii}
		\mathcal{Y}_p := \frac{\Theta_p^k}{A_p^k} - \omega^k, \quad \upsilon_p := \sum_{j=0}^{k-1} \frac{\Theta_p^j}{A_p^j} \wedge \omega^{k-1-j}, \quad \beta_p := \frac{c_1(L_p, h_p)}{A_p} - \omega.
	\end{equation}
	
	Note that $\upsilon_{p}$ is a positive $(k-1, k-1)$-current from Theorem \ref{indpw}. By the diophantine approximation (\ref{cur}) and (\ref{excur}), we obtain:
	\begin{equation}\label{algmani}
		\|\beta_p\|_{\mathscr{C}^0} \leq \frac{C_0}{A_p^a}, \quad \frac{\Theta_p}{A_p} - \omega = \beta_p + \frac{1}{2A_p} dd^c \log K_p + \frac{1}{A_p} dd^c \Bigg( \int_{s_p \in H^0(X, L_p)} \log\big|\langle a, u_p(x)\rangle\big|_{h_p} d\sigma_p(s_p) \Bigg).
	\end{equation}
	
	For $\phi \in \mathcal{D}^{n-k, n-k}(X)$, substituting (\ref{algmanii}) and (\ref{algmani}) yields
	\begin{equation}\label{mainesti}
		\langle \mathcal{Y}_p, \phi \rangle = \langle \Big(\frac{\Theta_p}{A_p} - \omega\Big) \wedge \upsilon_p, \phi \rangle = \int_X \upsilon_p \wedge \beta_p \wedge \phi + \int_X \frac{\log K_p}{2A_p} \upsilon_p \wedge dd^c \phi + \int_X \Bigg( \int_{H^0(X, L_p)} \frac{\log|\langle a, u_p\rangle|_{h_p}}{A_p} d\sigma_p \Bigg) \upsilon_p \wedge dd^c \phi.
	\end{equation}
	
	Continuing as in the proof of Proposition 3.5 of \cite{CLMM} and applying (\ref{hercomp}), (\ref{form2}), (\ref{antc}) and (\ref{m1.3}) where appropriate, we find
	\begin{equation*}
		\big| \langle \mathbb{E}[\widehat{Z}_{\Sigma^k_p}] - \omega^k, \phi \rangle \big| = \big| \langle \mathcal{Y}_p, \phi \rangle \big| \leq C'_n \|\phi\|_{\mathscr{C}^2} \mathrm{Vol}(X) \Big( \frac{\log A_p}{A_p} + \frac{C_p^{1/\alpha}}{A_p} + A_p^{-a} \Big),
	\end{equation*}
	for all $p \geq \max\{p_0, p_1, p'\}$, where $C'_n > 0$ is a constant depending only on the dimension $n$, completing the proof.
\end{proof}

\begin{thm}Let $(X,\omega)$  be a compact K\"{a}hler manifold of $\textup{dim}_{\mathbb{C}}X=n$ and let  $(L_{p},h_{p})_{p\geq 1},$ be a sequence of Hermitian holomorphic line bundles on $X$ with $\mathscr{C}^{2}$ metrics $h_{p}$. Assume that the conditions \textbf{(A)}, \textbf{(B)} and \textbf{(C1)-(C2)} hold.  If $$\ \sum_{p=1}^{\infty}\frac{C_{p}^{2/\alpha}}{A_{p}^{2}}< \infty,$$ then for $\sigma_{\infty}^{k}$-almost every sequence $\mathbf{\Sigma_{k}}=\left \{ \Sigma_{p}^{k} \right \} _{p\geq 1} \in \mathcal{H}_{\infty}^{k},$
	\begin{equation*}
		\big[\widehat{Z}_{\Sigma_{p}^{k}}\big]\longrightarrow \omega^k
	\end{equation*}
	in the weak* topology of currents as $p \rightarrow \infty.$
\end{thm}
\begin{proof}
	Fix $\phi \in \mathcal{D}^{n-k, n-k}(X),$ and pick $\mathbf{\Sigma_{k}}=\left \{ \Sigma_{p}^{k} \right \} _{p\geq 1} \in \mathcal{H}_{\infty}^{k}$. We argue by using the method in the proof of \cite[Corollary 1.3]{Shif}. Let us examine the non-negative random variables
	\begin{equation}\label{X_p}
		\mathcal{X}_{p}(\mathbf{\Sigma_{k}}):=\big \langle \big[\widehat{Z}_{\Sigma_{p}^{k}}\big] - \mathbb{E}\big[\widehat{Z}_{\Sigma_{p}^{k}}\big] , \phi  \big\rangle ^{2}\geq 0
	\end{equation}
	By appealing to the equivalent characterization of variance, notice that
	\begin{equation}\label{bep}
		\int_{\mathcal{H}_{\infty}^{k}}\mathcal{X}_{p}(\mathbf{\Sigma_{k}})d\sigma_{\infty}^{k}(\mathbf{\Sigma_{k}})=
		 \mathrm{Var}\big\langle \big[\widehat{Z}_{\Sigma_{p}^{k}}\big], \, \phi \big\rangle
	\end{equation}
	Using Theorem \ref{th1} along with the summability condition given by the hypothesis, we get
   \begin{equation}
		\sum_{p=1}^{\infty} \int_{\mathcal{H}_{\infty}^{k}}\mathcal{X}_{p}(\mathbf{\Sigma_{k}})d\sigma_{\infty}^{k}(\mathbf{\Sigma_{k}})= \sum_{p=1}^{\infty}  \mathrm{Var} \big\langle \big[\widehat{Z}_{\Sigma_{p}^{k}}\big], \, \phi \big\rangle< \infty
	\end{equation}
	By (\ref{bep}) above and invoking Beppo-Levi Theorem from the standard measure theory, we get
	\begin{equation}
		\int_{\mathcal{H}_{\infty}^{k}}\sum_{p=1}^{\infty} \mathcal{X}_{p}(\mathbf{\Sigma_{k}})d\sigma_{\infty}^{k}(\mathbf{\Sigma_{k}}) =  \sum_{p=1}^{\infty}  \mathrm{Var} \big\langle \big[\widehat{Z}_{\Sigma_{p}^{k}}\big], \, \phi \big\rangle < \infty
	\end{equation}
	This implies that, for $\sigma_{\infty}^{k}$- almost every sequence $\mathbf{\Sigma_{k}} \in \mathcal{H}_{\infty}^{k}$ of systems, the series $\sum_{p=1}^{\infty}\mathcal{X}_{p}(\mathbf{\Sigma_{k}})$ converges, leading to the conclusion that $\mathcal{X}_{p}\rightarrow 0$\, $\sigma_{\infty}^{k}$-almost surely. By definition (\ref{X_p}) of random variables $\mathcal{X}_{p}$ this also indicates that
	\begin{equation}
	\big \langle \big[\widehat{Z}_{\Sigma_{p}^{k}}\big] - \mathbb{E}\big[\widehat{Z}_{\Sigma_{p}^{k}}\big] , \phi  \big\rangle \rightarrow 0
		\end{equation}
$\sigma_{\infty}^{k}$-almost surely. Combining  this last information with Theorem \ref{exdist}, we conclude that for $\sigma_{\infty}^{k}$-almost every sequence,
\begin{equation}
	\big[\widehat{Z}_{\Sigma_{p}^{k}}\big]\longrightarrow \omega^k
\end{equation}
in the weak* topology of currents as $p\rightarrow \infty$.
\end{proof}

\section{Some Special Cases}\label{sec5}

 In \cite{BCM}, certain types of measures which satisfy the assumption \textbf{(C1)} and \textbf{(C2)} have been investigated as special cases. We will now provide some insights concerned with most of these measures in connection with Theorems \ref{th1} and \ref{th2}. The first two measures to be considered here will be the Gaussian and the Fubini-Study measures, both of which are unitary invariant measures that come with certain advantages in estimations.
 \subsection{Gaussian and Fubini-Study}In what follows, $\lambda_{n}$ represents the Lebesgue measure on $\mathbb{C}^{n}$ (identified with $\mathbb{R}^{2n}$). We will present the variance estimate simultaneously for both Gaussian and Fubini-Study cases, with detailed explanations provided for the Gaussian case as the computations are exactly the same. It turns out that, in these cases, the constants $C_{p}$ reduce to the ones independent of $p$ and Theorem \ref{th1} remains valid for every $\alpha \geq 1$. The standard Gaussian measure is precisely defined as follows, for $a=(a_{1}, \ldots, a_{n})\in \mathbb{C}^{n},$  \begin{equation} \label{gaussme}d\sigma_{n}(a)=\frac{1}{\pi^{n}}e^{-||a||^{2}}d\lambda_{n}(a),\end{equation}and the Fubini-Study measure on $\mathbb{C}\mathbb{P}^{n}\supset \mathbb{C}^{n}$  is defined as: \begin{equation}\label{fubstudy} d\sigma_{n}(a)=\frac{n!}{\pi^{n}}\frac{1}{(1+ ||a||^{2})^{n+1}}d\lambda_{n}(a).\end{equation}

  As for these two measures, we record two facts (Lemma 4.8, Lemma 4.10) from \cite{BCM}: Given that  $\sigma_{n}$ is the Gaussian measure, for every integer $n\geq 1$ and every $\alpha \geq 1$, we have \begin{equation}\label{gauss1}\int_{\mathbb{C}^{n}}{|\log{|\langle a, v \rangle|}|}^{\alpha}d\sigma_{n}(a)= 2 \int_{0}^{\infty}{r|\log{r}|^{\alpha} e^{-r^{2}} dr}, \,\,\,\,\forall v\in \mathbb{C}^{n},\,\,\,||v||=1;\end{equation} if $\sigma_{n}$ is the Fubini-Study, then for every integer $n \geq 1$ and every $\alpha \geq 1$ \begin{equation}\label{gauss2}\int_{\mathbb{C}^{n}}{|\log{|\langle a, v \rangle|}|}^{\alpha}d\sigma_{n}(a)= 2\int_{0}^{\infty}{\frac{r|\log{r}|^{\alpha}}{(1+r^{2})^{2}}dr}, \,\,\,\,\forall v\in \mathbb{C}^{n},\,\,\,||v||=1.\end{equation}As was remarked following Theorem \ref{th2} in the introduction, they are independent of the dimension $n$. For these two special measures, Theorem \ref{th1} becomes the following.

     \begin{thm}\label{Vard} Under the condition \textbf{(A)}, assume that for every $j=1,\ldots,k$,
the linear system $H^{0}(X, L_p)$ is base-point free, i.e. $\operatorname{Bs}(H^{0}(X, L_p))=\varnothing$ for large values of $p$. Let $\sigma^{k}_{d_{p}}$ be the product Gaussian (Fubini-Study) measure on $H^{0}(X, L_{p})^{k}$ given by (\ref{gaussme}) (respectively,  (\ref{fubstudy})). Then for any  $\phi\in \mathcal{D}^{n-k, n-k}(X)$, one gets \begin{equation}\label{mainvar-gf}  \mathrm{Var}\big\langle \big[\widehat{Z}_{\Sigma_{p}^{k}}\big],\, \phi \big\rangle \leq \frac{1}{A^{2}_{p}}\Lambda_{k}\,B^{2}_{\phi} \,\mathrm{Vol}(X)^{2}, \end{equation}where  $\Lambda_{k}=2^{k-1}\int_{0}^{\infty}{r|\log{r}|^{\alpha} e^{-r^{2}} dr}$ (respectively, $\Lambda_{k}=2^{k-1}\int_{0}^{\infty}{\frac{r|\log{r}|^{\alpha}}{(1+r^{2})^{2}}dr}$).  \end{thm}

We infer from Theorem \ref{th2} the subsequent theorem

 \begin{thm} \label{GauFub}
With the same assumptions of Theorem \ref{Vard} and the assumption \textbf{(B)}, let $\sigma_{p}$ be the Gaussian (Fubini-Study) measure on   $H^{0}(X, L_{p})\simeq \mathbb{C}^{d_{p}}$ given by (\ref{gaussme}) (respectively, \ref{fubstudy}). Then, for $1\leq k \leq \textup{dim}_{\mathbb{C}} X$
	\begin{equation}
		\mathbb{E}\big[\widehat{Z}_{\Sigma_{p}^{k}}\big]\longrightarrow \omega^k
	\end{equation}
	in the weak* topology of currents as $p\rightarrow \infty.$ In addition, if $\sum_{p=1}^{\infty}\frac{1}{A^{2}_{p}}< \infty,$  then for $\sigma^{k}_{\infty}-$almost every sequence $\{\Sigma_{p}^{k}\}\in \mathcal{H}^{k}_{\infty}$ we have
	\begin{equation}
		\big[\widehat{Z}_{\Sigma_{p}^{k}}\big]\longrightarrow \omega^k
	\end{equation}
	in the weak* topology of currents as $p \rightarrow \infty.$
 \end{thm}

When we consider the prequantum line bundle setting, where $(L_{p}, h_{p})=(L^{\otimes p}, h^{\otimes p})$ and  $c_{1}(L, h)= \omega \,$ in Theorem \ref{Vard} and Theorem \ref{GauFub}, we recover the results of Shiffman-Zelditch (\cite{SZ99}).

Theorem~\ref{GauFub} is proved using an entirely different approach in \cite[Theorem 0.4]{CLMM}, which relies on the Fubini-Study probability measures and the framework of meromorphic transforms from complex dynamics, originally developed by Dinh and Sibony in \cite{DS06}. In this setting, they define certain exceptional sets, including the set where the currents of integration associated with zero sets of holomorphic sections are not well-defined. Whereas we impose a summability condition on a different quantity involving the (uniform) upper bound coming from the logarithmic moment condition \textbf{(C2)} and normalization, their key assumption in \cite[Theorem 0.4]{CLMM} is the summability of the measures of these exceptional sets, a condition that is intrinsically linked to the nature of the exceptional sets themselves.

\subsection{Area Measure of Spheres} Let $\mathcal{A}_{n}$ be the surface area measure on the unit sphere $S^{2n-1}$ in $\mathbb{C}^{n},$ given by $\mathcal{A}_{n}(S^{2n-1})=\frac{2\pi^{n}}{(n-1)!}.$ Let us consider the following probability measure on $S^{2n-1}$
\begin{equation}\label{sphere}
	\sigma_{n}=\frac{1}{\mathcal{A}_{n}(S^{2n-1})}\mathcal{A}_{n}
\end{equation}
 Given that  $\sigma_{n}$ is the normalized area measure on the unit sphere, by Lemma 4.11 from \cite{BCM}, for every $\alpha \geq 1$, there exists a constant $C_{\alpha}>0$ such that for every integer $n\geq2,$ we have: \begin{equation}\int_{\mathbb{C}^{n}}{|\log{|\langle a, v \rangle|}|}^{\alpha}d\sigma_{n}(a)\leq C_{\alpha}\,(\log n)^{\alpha}, \,\,\,\,\forall v\in \mathbb{C}^{n},\,\,\,||v||=1;\end{equation}

One should remark that, in this specific case, even though the measure is unitary invariant, the aforementioned upper bound is not a universal constant.

Now, due to the fact  that $C_{\alpha}\,(\log d_{p})^{\alpha}\leq C_{\alpha} \,((n+2)\log A_{p})^{\alpha}$ for sufficiently large $p$, utilizing Theorem \ref{th1} with $C_{p}=C_{\alpha} \,((n+2)\log A_{p})^{\alpha}$ leads to the following variance estimate

\begin{thm} Under the same assumptions of Theorem \ref{Vard}, let $\sigma_{p}:=\sigma_{d_{p}}$ be the normalized area measure on the unit sphere of  $H^{0}(X, L_{p})\simeq \mathbb{C}^{d_{p}} $ given by (\ref{sphere}). Then for any  $\phi\in \mathcal{D}^{n-k, n-k}(X)$ and sufficiently large $p$, one has 
\begin{equation}
		\mathrm{Var}\langle [\widehat{Z}_{\Sigma_{p}^{k}}], \phi \rangle   \leq \Big(\frac{\log{A_{p}}}{A_{p}}\Big)^{2} \Lambda_{k,n,\alpha} B_{\phi}^{2}\,
\end{equation}
where  $\Lambda_{k,n,\alpha}=((n+2)\mathrm{Vol}(X)\ C_{\alpha}^{1/\alpha})^{2} D_n$ is a positive constant. \end{thm}

Consequently, we have
\begin{thm} Under the same assumptions of Theorem \ref{GauFub} , let $\sigma_{p}$ be the normalized area measure on the unit sphere of  $H^{0}(X, L_{p})\simeq \mathbb{C}^{d_{p}}$ given by (\ref{sphere}). Then, for $1\leq k \leq \textup{dim}_{\mathbb{C}} X$
	\begin{equation}
		\mathbb{E}\big[\widehat{Z}_{\Sigma_{p}^{k}}\big]\longrightarrow \omega^k
	\end{equation}
	in the weak* topology of currents as $p\rightarrow \infty.$ In addition, if $\sum_{p=1}^{\infty}\big(\frac{\log A_{p}}{A_{p}}\big)^{2}< \infty,$  then for $\sigma_{\infty}^{k}$-almost every sequence $\{\Sigma_{p}^{k}\}_{p\geq 1}\in \mathcal{H}_{\infty}^{k}$ we have
	\begin{equation}
		\big[\widehat{Z}_{\Sigma_{p}^{k}}\big]\longrightarrow \omega^k
	\end{equation}
	in the weak* topology of currents as $p \rightarrow \infty.$
\end{thm}

\subsection{Random Holomorphic Sections with i.i.d. Coefficients} In this context, we examine the probability space $(H^{0}(X,L_{p}), \sigma_{p})$
where $\sigma_{p}$ is the product probability measure induced by the probability distribution law $\mathbb{P}$ governing the i.i.d. random coefficients  $a_{j}^{p}$ in the representation (\ref{rep}). This distribution possesses a bounded density $\psi:\mathbb{C}\rightarrow [0, M],$ and satisfies  the property that there exist constants $\epsilon>0\,\text{and} \, \delta>1,$ such that
\begin{equation}
	\mathbb{P}\big(\{z\in \mathbb{C}: \log|z|>R\}\big)\leq \frac{\epsilon}{R^{\delta}}, \, \, \text{for all } R\geq 1.
\end{equation}
This particular density type has been investigated in \cite{Bay16} and \cite{BCM}, and it encompasses distributions such as the real or complex Gaussian distributions.
 Given such a measure $\sigma_{p}$ on $H^{0}(X, L_{p})$, according to  Lemma 4.15 of \cite{BCM} we have, for any $1\leq \alpha <\delta:$
 \begin{equation}\int_{\mathbb{C}^{d_{p}}}{|\log{|\langle a, v \rangle|}|}^{\alpha}d\sigma_{p}(a)\leq B\,d_{p}^{\alpha /\delta}, \,\,\,\,\forall v\in \mathbb{C}^{d_{p}},\,\,\,||v||=1;\end{equation}
where $B=B(M, \epsilon, \delta, \alpha)>0.$ In our present setting, for $p$ sufficiently large, $d_{p}\leq M_{0}\,\mathrm{Vol}(X)\,A_{p}^{n}.$ Using  Theorem \ref{th1} with $C_{p}=DA_{p}^{n\frac{\alpha}{\delta}},$ where $D=(M_{0}\,\mathrm{Vol}(X))^{\alpha/\delta} B$ we obtain
\begin{thm} Under the same assumptions of Theorem \ref{Vard}, if $\sigma_{p}$ is the probability measure on  $H^{0}(X, L_{p})\simeq \mathbb{C}^{d_{p}}$ defined as above. Then for any  $\phi\in \mathcal{D}^{n-k, n-k}(X)$ and sufficiently large $p$, one has \begin{align*}
		\mathrm{Var}\langle [\widehat{Z}_{\Sigma_{p}^{k}}], \phi \rangle   \leq \Big(\frac{1}{A_{p}^{1-n/\delta}}\Big)^{2} \big(D^{1/\alpha}\,\mathrm{Vol}(X)B_{\phi}\big)^{2} D_n
	\end{align*}where  $D=(M_{0}\mathrm{Vol}(X))^{\alpha/\delta}B$ is a positive constant.   \end{thm}
As a consequence,  we have the equidistribution result
\begin{thm} \label{decay} Let $(L_{p}, h_{p})_{p\geq 1}$,\, $(X, \omega)$ be as in Theorem \ref{GauFub}. Assume that $\sigma_{p}$ is the probability measure on $H^{0}(X, L_{p})$ defined as above. If $\delta>n,$ then for $1\leq k \leq \textup{dim}_{\mathbb{C}} X$
	\begin{equation}
		\mathbb{E}\big[\widehat{Z}_{\Sigma_{p}^{k}}\big]\longrightarrow \omega^k
	\end{equation}
	in the weak* topology of currents as $p\rightarrow \infty.$ In addition, if $\sum_{p=1}^{\infty}\frac{1}{A_{p}^{2-2n/\delta}}< \infty,$ where $\delta > 2n$,  then almost surely
	\begin{equation}
		\big[\widehat{Z}_{\Sigma_{p}^{k}}\big]\longrightarrow \omega^k
	\end{equation}
	in the weak* topology of currents as $p \rightarrow \infty.$
\end{thm}

\subsection{Locally moderate measures}

Consider a complex manifold $X$ and a positive measure $\sigma$ on $X$. In accordance with \cite{DNS}, we define $\sigma$ as a locally moderate measure if, for any open set $U \subset X$, any compact set $K \subset U$, and any compact family $\mathscr{F}$ of plurisubharmonic functions on $U$, there exist positive constants $M$ and $\beta$ such that
\begin{equation}
	\int_{K}e^{-\beta \varphi }d\sigma \leq M, \ \ \text{for all } \varphi \in \mathscr{F}.
\end{equation}

It is evident that $\sigma$ does not charge pluripolar sets. Furthermore, we remark that important examples of such measures arise from the Monge-Ampère measures associated with Hölder continuous plurisubharmonic functions, for more details in this direction see \cite{DNS}. According to \cite[Lemma 4.16]{BCM}, if $\sigma_{p}$ is a locally moderate probability measure with compact support in $\mathbb{C}^{d_{p}}\cong H^{0}(X,L_{p}),$ then for every $\alpha\geq 1$
 \begin{equation}\int_{\mathbb{C}^{d_{p}}}{|\log{|\langle a, v \rangle|}|}^{\alpha}d\sigma_{p}(a)\leq \Lambda_{p} R_{p}^{2\beta_{p}}, \,\,\,\,\forall v\in \mathbb{C}^{d_{p}},\,\,\,||v||=1;\end{equation}
where $\Lambda_{p}, \beta_{p}>0$ are positive constants and $R_{p}\geq 1$ such that $\|a\|\leq R_{p} $ for all $a\in \mathrm{supp} \ \sigma_{p}.$ Continuing in the same manner as the previous examples, we deduce the following results.

\begin{thm} Under the same assumptions of Theorem \ref{Vard}, if $\sigma_{p}$ is a locally moderate probability measure with compact support in $\mathbb{C}^{d_{p}}\cong H^{0}(X,L_{p})$. Then for any  $\phi\in \mathcal{D}^{n-k, n-k}(X)$ and sufficiently large $p$, one has \begin{align*}
		\mathrm{Var}\langle [\widehat{Z}_{\Sigma_{p}^{k}}], \phi \rangle  \leq \frac{(\Lambda_{p} R_{p}^{2\beta_{p}})^{2/\alpha}}{A_{p}^{2}}\,(\mathrm{Vol}(X)B_{\phi}\, )^{2} D_n
	\end{align*}where $\Lambda_{p}, \beta_{p}>0$ are positive constants and $R_{p}\geq 1$ .   \end{thm}

\begin{thm}Let $(L_{p}, h_{p})_{p\geq 1}$,\, $(X, \omega)$ be as in Theorem \ref{GauFub}. Assume that $\sigma_{p}$ is the locally moderate probability measure on $H^{0}(X, L_{p})$ defined as above.
	
	\begin{itemize}
		\item[(i)] If $\ \lim_{p \rightarrow \infty}\frac{(\Lambda_{p} R_{p}^{2\beta_{p}})^{1/\alpha}}{A_{p}}=0$ then for $1\leq k \leq \textup{dim}_{\mathbb{C}} X$
		\begin{equation*}
			\mathbb{E}\big[\widehat{Z}_{\Sigma_{p}^{k}}\big]\longrightarrow \omega^k
		\end{equation*}
		in the weak* topology of currents as $p\rightarrow \infty.$
		\item[(ii)] If $\ \sum_{p=1}^{\infty}\frac{(\Lambda_{p} R_{p}^{2\beta_{p}})^{2/\alpha}}{A_{p}^{2}}< \infty,$  then for $\sigma^{k}_{\infty}$-almost all $\{\Sigma_{p}^{k}\}_{p\geq 1}\in \mathcal{H}_{\infty}^{k}$
		\begin{equation*}
			\big[\widehat{Z}_{\Sigma_{p}^{k}}\big]\longrightarrow \omega^k
		\end{equation*}
		in the weak* topology of currents as $p \rightarrow \infty.$
	\end{itemize}
\end{thm}

\section{Appendix}

\subsection{Bertini-type genericity and proper intersections}\label{app-bertini}
Adapting the proof of Proposition 3.2 in \cite{CMN}, we prove a probabilistic Bertini-type theorem for products of (possibly distinct) probability measures that do not charge pluripolar sets. As a consequence, under this product measure, the intersection current is well-defined almost surely.
\begin{defn}\label{gpdef}
	The analytic subsets $A_{1},\ldots, A_{m}, \ m\leq n$, of a compact complex manifold $X$ of dimension $n$ are said to be in general position if $\mathrm{codim}A_{i_{1}}\cap \ldots \cap A_{i_{k}}\geq k $ for every $1\leq k\leq m$ and $1\leq i_{1}<\ldots<i_{k}\leq m.$
\end{defn}

\begin{prop} \label{bertiniseveral}
	Let $L_{k} \to X $ be holomorphic line bundles over a compact complex manifold $X$ with $\dim_{\mathbb{C}}{X}=n$, where $1 \leq k \leq m \leq n$. Suppose that:
	
	\begin{enumerate}
		\item[(i)] $V_k$ is a subspace of $H^0(X, L_{k})$ with a basis $\{S_{k, 1}, \ldots, S_{k, d_{k}}\}$, and the base loci $ \operatorname{Bs}(V_{k}) = \{x\in X: S_{k, 1}(x) = \ldots = S_{k, d_{k}}(x) = 0\}$ are all empty.
		\item[(ii)] $Z(t^{k}) = \{ x \in X : \sum_{j=1}^{d_{k}} {t_j S_{k, j}(x) = 0} \} $, where $t^{k} = (t_1, \ldots, t_{d_{k}}) \in \mathbb{C}^{d_{k}}$.
	\end{enumerate}
	
	If $\sigma=\sigma_{1} \times \cdots \times \sigma_{m}$ is the product probability measure on $\mathbb{C}^{d_{1}} \times \cdots \times \mathbb{C}^{d_{m}}$, where each probability measure $\sigma_k,\,k=1,\ldots, m$, satisfies \textbf{(C1)}, then the analytic sets $Z(t^{1}),  \ldots , Z(t^{m})$ are in general position for $\sigma$-almost every $(t^{1}, \ldots, t^{m}) \in \mathbb{C}^{d_{1}} \times \cdots \times \mathbb{C}^{d_{m}}$.
\end{prop}

\begin{proof}
Given $1\le l_1<\cdots<l_k\le m$, let $\sigma_{l_1\cdots l_k}=\sigma_{l_1}\times\cdots\times \sigma_{l_k}$ be the product probability measure on $\mathbb{C}^{d_{l_1}}\times\cdots\times\mathbb{C}^{d_{l_k}}$.
For $1\le k\le m$, define the following sets
$$ \label{codimm}
H_k:=\Bigl\{(t^{l_1},\ldots,t^{l_k})\in \mathbb{C}^{d_{l_1}}\times\cdots\times\mathbb{C}^{d_{l_k}}:\
\dim\bigl(Z(t^{l_1})\cap\cdots\cap Z(t^{l_k})\bigr)\le n-k\Bigr\}.
$$ We prove $\sigma_{l_1\cdots l_k}(H_{k})=1$ by induction on $k$ for every set $H_{k}$ with $1\le l_1<\cdots<l_k\le m$, so it will suffice to consider the case $\{l_{1}, \ldots, l_{k}\}=\{1, \ldots, k\}$. We start with the case $k=1$. We immediately have that, whatever $t^{1} \in \mathbb{C}^{d_{1}}$ is chosen, $\{t^{1} \in \mathbb{C}^{d_{1}}: \dim{Z(t^{1})} \leq n-1\}= \mathbb{C}^{d_{1}} \backslash \{0\}$ since any single analytic subset is always in general position. Suppose that $\sigma_{1,\ldots, k}(H_{k})=1$ for all $H_{k}$ defined as in (\ref{codimm}). Let
	\begin{equation}\label{k1}
		H_{k+1}=\{(t^{1}, \ldots, t^{k+1}) \in \mathbb{C}^{d_{1}} \times \ldots \times \mathbb{C}^{d_{k+1}}: \dim{Z(t^{1}) \cap \ldots \cap Z(t^{k+1})} \leq n-k-1\}.
	\end{equation}We need to show that $\sigma_{1, \ldots, k+1}(H_{k+1})=1$, so we show that the $\sigma_{1, \ldots, k+1}$-measure of the complement set $H^{c}_{k+1}$ is zero. First, let us fix $t=(t^{1}, \ldots, t^{k}) \in H_{k}$. Define $Z(t):= Z(t^{1}) \cap \ldots \cap Z(t^{k})$ and  \begin{equation} \label{Gt}G(t):=\{t^{k+1} \in \mathbb{C}^{d_{k+1}}: \dim{Z(t) \cap Z(t^{k+1})} \geq n-k\}. \end{equation} It is enough to prove that $\sigma_{d_{k+1}}(G(t))=0$. These sets $G(t)$ are called the slices of the set $H^{c}_{k+1}$. Let \begin{equation}Z(t)= \bigcup_{k=1}^{N_{0}}{E_{l} \cup Y}, \end{equation} where $E_{l}$ are, as in the case $k=1$, the irreducible components of $Z(t)$ with $\dim{E_{l}}=n-k$ and $\dim{Y}=n-k-1$. If $t^{k+1} \in G(t)$, then $Z(t) \cap Z(t^{k+1})$ is an analytic subset of $Z(t)$ with $\dim{Z(t) \cap Z(t^{k+1})}=n-k$, and this gives that there is some $l\in \{1, \ldots N_{0}\}$ such that\begin{equation}\label{cdmk}
		E_{l} \subset Z(t) \cap Z(t^{k+1}).
	\end{equation}
	Hence we have \begin{equation}\label{codec}
		G(t)=\bigcup_{l=1}^{N_{0}}{A_{l}(t)}, \,\,A_{l}(t):=\{t^{k+1} \in \mathbb{C}^{d_{k+1}}: E_{l} \subset Z(t^{k+1})\}.
	\end{equation}Now we see that all sets $A_{l}(t)$ are proper linear subspaces of $\mathbb{C}^{d_{k+1}}$. Indeed, if some were not so, then we would deduce that, for some $l_{0} \in \{1, 2, \ldots, N_{0}\}$, $$\sum_{j=1}^{d_{k+1}}{t^{k+1}_{j}S_{k+1, j}|_{E_{l}}}=0$$ for all $t^{k+1}=(t^{k+1}_{1}, \ldots, t^{k+1}_{d_{k+1}}) \in A_{l_{0}}(t)=\mathbb{C}^{d_{k+1}}$, which would imply that $E_{l} \subset \operatorname{Bs}(V_{k+1})$, i.e., that the base locus was non-empty, contradicting our assumption that the base locus $\operatorname{Bs}(V_{k+1})$ is empty, and so they are pluripolar. By \textbf{(C1)}, $\sigma_{d_{k+1}}(A_{l}(t))=0$, and so $\sigma_{d_{k+1}}(G(t))=0$, finishing the proof. \end{proof}

Instead of the proper linear subspace argument above at the end of the proof, one can go on as in [12, Proposition 3.2]: By the same reasoning as in the same paragraph, not all sections are zero on $E_{l}$, and so we may then assume that $S_{k+1, d_{k+1}} \neq 0$ on $E_{l}$. Now, for any $(t^{k+1}_{1}, \ldots, t^{k+1}_{d_{k+1}-1}) \in \mathbb{C}^{d_{k+1}-1}$, there exist at most one $h\in \mathbb{C}$ such that $(t^{k+1}_{1}, \ldots, t^{k+1}_{d_{k+1}-1}, h) \in A_{l}(t)$, otherwise, if there exist two different elements $h, h' \in \mathbb{C}$ with this property, we have \begin{align*}\label{sonic}
		t^{k+1}_{1} S_{k+1, 1} + \ldots + t^{k+1}_{d_{k+1}-1} S_{k+1, d_{k+1}-1} + h S_{k+1, d_{k+1}}=0 \\
		t^{k+1}_{1} S_{k+1, 1} + \ldots + t^{k+1}_{d_{k+1}-1} S_{k+1, d_{k+1}-1} + h' S_{k+1, d_{k+1}}=0,
	\end{align*}which implies that $S_{k+1, d_{k+1}} \equiv 0$ on $E_{l}$, which is a contradiction. Thus, $\sigma_{d_{k+1}}(A_{l}(t))=0$. This implies that $\sigma_{d_{k+1}}(G(t))=0$, which concludes the proof.

\vspace{3mm}

In the setting of Section \ref{S2}, with the assumption that there exists $p_{0} \in \mathbb{N}$ such that $\operatorname{Bs}(\mathcal{S}_{p, j}) =\emptyset$ for all $p \geq p_{0}$ and $j=1, 2, \ldots, k$, by using the arguments from Lemma 3.1 in \cite{CLMM} based on the results of Demailly (Corollary 4.11 and Proposition 4.12 in \cite{Dem12}), as a result of Proposition \ref{bertiniseveral}, we arrive at the following proposition

\begin{prop} \label{wdberti}
	There exists $p_{0} \in \mathbb{N}$ such that for all $p \geq p_{0}$,
	\begin{itemize}
		\item [(i)] The analytic subvarieties $Z_{s^{1}_{p}}, \ldots, Z_{s^{m}_{p}}$ are all in general position for $\sigma^{m}_{p}$-almost all $(s^{1}_{p}, \ldots, s_{p}^{m}) \in \mathcal{S}_{p, 1} \times \ldots \times \mathcal{S}_{p, m}$.
		\item [(ii)] If the assumption \textbf{(A)} holds, then for $\sigma^{k}_{p}$-almost every $\Sigma^{k}_{p}=(s^{j_{1}}_{p}, \ldots, s^{j_{l}}_{p})$, the analytic subvariety $Z_{s^{j_{1}}_{p}} \cap \ldots \cap Z_{s^{j_{k}}_{p}}$ is of pure dimension $n-k$ for each $1\leq k \leq m$ and $1\leq j_{1}<\ldots<j_{k}\leq m$.
		\item [(iii)] The intersection current $[Z_{\Sigma^{m}_{p}}]:=[Z_{s^{1}_{p}}]\wedge \cdots \wedge[Z_{s_{p}^{m}}]$ is well-defined and is equal to the current of integration with multiplicities over the complete intersection $Z_{\Sigma^{m}_{p}}$.	
	\end{itemize}
\end{prop}

\textbf{Acknowledgement:} We thank Turgay Bayraktar for his interest in our work and comments on the first version of this paper. We sincerely thank the referee for the time and effort devoted to a meticulous evaluation of our manuscript, and for the comments, suggestions, and corrections that improved the clarity of the present paper.

\end{document}